\documentclass[10pt,a4paper]{amsart}

\usepackage{amsthm,amsmath,amssymb,amsxtra,amscd,enumerate}
\usepackage[all]{xy}
\usepackage{longtable}
\usepackage{float}
\usepackage{color}
\usepackage{hyperref}

\numberwithin{equation}{section}

\theoremstyle{plain}
\newtheorem{thm}{Theorem}[section]
\newtheorem{prop}[thm]{Proposition}
\newtheorem{lem}[thm]{Lemma}

\newtheorem{conj}[thm]{Conjecture}
\theoremstyle{definition}
\newtheorem{de}[thm]{Definition}
\newtheorem{rem}[thm]{Remark}
\newtheorem{set}[thm]{Setting}

\newcommand{\bb}{\mathbb}
\newcommand{\cal}{\mathcal}

\DeclareMathOperator{\Hom}{Hom}

\DeclareMathOperator{\ad}{ad}

\DeclareMathOperator{\pr}{pr}

\DeclareMathOperator{\Ann}{Ann}

\title{Discrete branching laws for minimal holomorphic representations}

\author{Jan M\"ollers \and Yoshiki Oshima}

\address{Institut for Matematiske Fag, Aarhus Universitet, Ny Munkegade 118, 8000 Aarhus C, Denmark}
\email{moellers@imf.au.dk}

\address{School of Mathematics, Institute for Advanced Study, Einstein Drive, Princeton, NJ 08540, USA}
\email{yoshiki.oshima@ipmu.jp}

\begin{document}

\maketitle

\begin{abstract}
We find the explicit branching laws for the restriction of minimal holomorphic representations to symmetric subgroups in the case where the restriction is discretely decomposable. For holomorphic pairs the minimal holomorphic representation decomposes into a direct sum of lowest weight representations which is made explicit. For non-holomorphic pairs the restriction is shown to be irreducible and identified with a known representation.\\
We further study a conjecture by Kobayashi on the behaviour of associated varieties under restriction and confirm this conjecture in the setting of this paper.
\end{abstract}

\section{Introduction}\label{sec:intro}

Let $G$ be a connected, simply connected real reductive Lie group with Lie algebra $\frak{g}$.
Let $\theta$ be a Cartan involution of $G$.
Write $K=G^\theta$ for the subgroup of $\theta$-fixed elements so that $K/(K \cap Z)$ is a compact group, $Z$ denoting the center of $G$.
Write $\frak{g}=\frak{k}+\frak{p}$ for the corresponding Cartan decomposition.

\begin{de}
The real reductive Lie algebra $\frak{g}=\frak{k}+\frak{p}$
is called of {\it Hermitian type} 
and the symmetric pair $(\frak{g},\frak{k})$ is a 
{\it Hermitian symmetric pair}
if there exists an element $z\in\frak{k}$ such that 
 $\ad(z)=0$ on $\frak{k}$ and $\ad(z)^2=-1$ on $\frak{p}$.
\end{de}

Suppose that $\frak{g}$ is of Hermitian type. 
Put $z':=-\sqrt{-1}z\in\frak{k}_\bb{C}$ and
 let $\frak{p}_+$ (resp.\ $\frak{p}_-$) be the eigenspace
 in the complexification $\frak{p}_\bb{C}$ of
 $\ad(z')$ to the eigenvalue $1$ (resp.\ $-1$).
Then we get a decomposition $\frak{p}_\bb{C}=\frak{p}_+ + \frak{p}_-$.
The element $z$ endows the Riemannian symmetric space $G/K$ with a complex structure by choosing $\frak{p}_+$ as the holomorphic tangent space at the base point. 

\begin{de}
Let $\frak{g}$ be a reductive Lie algebra of Hermitian type as above.
An irreducible $(\frak{g},K)$-module $V$ is called
 a {\it highest weight module} if $V^{\frak{p}_+}=\Ann_{\frak{p}_+}(V)\neq 0$, namely,
 there exists a non-zero vector $v\in V$ that is annihilated by
 $\frak{p}_+$.
Similarly, $V$ is called a {\it lowest weight module}
 if $V^{\frak{p}_-}\neq 0$.
\end{de}

We note that highest (resp.\ lowest) weight modules belong to the category ${\cal O}_{\frak{q}}$ for the parabolic subalgebra $\frak{q}=\frak{k}_\bb{C}+\frak{p}_+$ (resp.\ $\frak{q}=\frak{k}_\bb{C}+\frak{p}_-$).
Let $\frak{t}$ be a Cartan subalgebra of $\frak{k}$
 and choose a positive system $\Delta^+(\frak{k}_\bb{C},\frak{t}_\bb{C})$ of $\Delta(\frak{k}_\bb{C},\frak{t}_\bb{C})$.
For a dominant integral weight $\lambda\in\frak{t}_\bb{C}^*$
 we denote the irreducible representation of $K$
 with highest weight $\lambda$ by $F^{\frak{k}}(\lambda)$.
We let $\frak{p}_-$ act as zero on $F^{\frak{k}}(\lambda)$ and put
\begin{align*}
N^{\frak{g}}(\lambda)=U(\frak{g}_\bb{C})\otimes_{U(\frak{k}_\bb{C}+\frak{p}_-)}
 F^{\frak{k}}(\lambda).
\end{align*}
The $(\frak{g},K)$-module $N^{\frak{g}}(\lambda)$ has a unique 
 irreducible quotient $L^{\frak{g}}(\lambda)$.
Then $L^{\frak{g}}(\lambda)$ is a lowest weight $(\frak{g},K)$-module and
 all the irreducible lowest weight $(\frak{g},K)$-modules arise in this way.

A $(\frak{g},K)$-module is said to be unitarizable if it admits a Hermitian
 inner product with respect to which
 $\frak{g}$ acts by skew-Hermitian operators.
The unitarizable highest (or lowest) weight $(\frak{g},K)$-modules were independently
 classified by Enright--Howe--Wallach~\cite{EHW} and Jakobsen~\cite{Jak}.
Suppose that $\frak{g}$ is simple of Hermitian type and
 let $\zeta\in\sqrt{-1}\frak{t}^*$ be a weight such that
 $\zeta(z')>0$ and $\zeta([\frak{k},\frak{k}]\cap\frak{t})=0$.
Take a weight $\lambda_0\in\sqrt{-1}\frak{t}^*$ which is integral dominant for
 $\frak{k}$.
Then there exist numbers $a\in\bb{R}$, $c\in\bb{R}_{>0}$,
 $r\in \bb{Z}_{>0}$ such that
 $L^{\frak{g}}(\lambda_0+x\zeta)$ for $x\in\bb{R}$ is unitarizable if and only if 
 $x=a, a-c,\dots, a-(r-1)c$ or $x$ lies in the half-line $(a,\infty)$.
If $\lambda_0=0$, then $a-(r-1)c=0$ and $r$ equals the 
 real rank of $\frak{g}$.
Therefore, $L^{\frak{g}}(x\zeta)$ for $x\in\bb{R}$ is unitarizable if and only if 
 $x=0,c,\dots,(r-1)c$, or $x\in((r-1)c,\infty)$.

\begin{de}
Let $\frak{g}$ be a simple Lie algebra of Hermitian type. When the real rank of $\frak{g}$ is greater than $1$, we call $L^{\frak{g}}(c\zeta)$ the {\it minimal holomorphic representation}.
\end{de}

If $\frak{g}$ is not of type $A$, the minimal holomorphic representation is a minimal representation of $G$, namely, the annihilator ideal in $U(\frak{g}_\bb{C})$ is the Joseph ideal. In particular, it attains the smallest Gelfand--Kirillov dimension (this is also true for $\frak{g}$ of type $A$). 
We note that for $\frak{g}=\frak{sp}(n,\bb{R})$, the minimal holomorphic representation is isomorphic to the even part of the so-called metaplectic representation (also referred to as the oscillator representation or the Segal--Shale--Weil representation).

The argument in the proof of \cite[II, Theorem 5.10]{Wal79}
 gives the $K$-type decomposition of $L^{\frak{g}}(c\zeta)$:
\begin{align}
\label{eq:Ktype}
L^{\frak{g}}(c\zeta)|_{\frak{k}}\simeq
 \bigoplus_{k=0}^\infty F^{\frak{k}}(c\zeta+k\beta),
\end{align}
where $\beta$ is the highest root in $\frak{p}_+$.

\medskip

The restriction of the minimal holomorphic representation to non-compact subgroups has been studied before in some cases. The restriction of the metaplectic representation to a dual pair of subgroups is particularly well-studied in connection with Howe's correspondence (see e.g. \cite{KaVe} and references therein). For several other settings, explicit branching laws were obtained in both the discretely decomposable case and the non-discretely decomposable case (see \cite{BZ94,MS,PZ,Sek13,Sep07i,Sep07ii,Zha01}).

Our aim in this article is to give the explicit branching laws of minimal holomorphic representations for symmetric pairs when the restriction is discretely decomposable.  We take an involution $\sigma$ on $G$ which commutes with $\theta$ and consider the symmetric pair $(G,G^\sigma)$, where $G^\sigma:=\{g\in G:\sigma(g)=g\}$.

\begin{de}\label{de:holpair}
Suppose that $\frak{g}$ is a reductive Lie algebra of Hermitian type. 
We say a symmetric pair $(\frak{g},\frak{g}^\sigma)$ is of {\it holomorphic type} if $\sigma z=z$, or equivalently if $\sigma$ induces a holomorphic involution  on the Hermitian symmetric space $G/K$.
\end{de}

If $(\frak{g},\frak{g}^\sigma)$ is of holomorphic type,
 $\frak{g}^\sigma$ is of Hermitian type with the same element $z$, and the natural
 embedding $G^\sigma/K^\sigma \to G/K$ is a holomorphic map.

A systematic study of discretely decomposable restrictions was initiated by Kobayashi~\cite{Kob94, Kob98i, Kob98ii}.  For an irreducible $(\frak{g},K)$-module $V$, the restriction $V|_{\frak{g}^\sigma}$ is said to be {\it discretely decomposable} if it is a sum of $\frak{g}^\sigma$-modules of finite length.  When $V$ is unitarizable, $V|_{\frak{g}^\sigma}$ is discretely decomposable if and only if it is isomorphic to a direct sum of irreducible $\frak{g}^\sigma$-modules (see \cite[Lemma 1.3]{Kob98ii}).  In his series of papers \cite{Kob94, Kob98i, Kob98ii}, Kobayashi obtained several criteria for the discrete decomposability. When a $(\frak{g},K)$-module $V$ has the smallest Gelfand--Kirillov dimension, one criterion for the discrete decomposability of $V|_{\frak{g}^\sigma}$ becomes simple (\cite[Theorem 4.10]{KO}) and it was used to obtain a complete list of symmetric pairs $(\frak{g},\frak{g}^\sigma)$ such that $V|_{\frak{g}^\sigma}$ is discretely decomposable, see \cite[Theorem 5.1]{KO}.  Since the minimal holomorphic representations attain the smallest Gelfand--Kirillov dimensions, the list gives all the symmetric pairs with discretely decomposable restrictions $L^{\frak g}(c\zeta)|_{\frak{g}^\sigma}$.  Let us recall the classification of these pairs $(\frak{g},\frak{g}^\sigma)$.

For a holomorphic pair $(\frak{g},\frak{g}^\sigma)$, any highest (or lowest) weight module is known to be discretely decomposable (\cite[Theorem 7.4]{Kob98iii}). For a non-holomorphic pair $(\frak{g},\frak{g}^\sigma)$, the restriction $L^{\frak g}(c\zeta)|_{\frak{g}^\sigma}$ is discretely decomposable if and only if $(\frak{g},\frak{g}^\sigma)$ is one of the following:
\begin{equation*}
\begin{aligned}
 &(\frak{su}(2m,2n),\frak{sp}(m,n)), && (m,n\geq1), &&(\frak{so}(2,n),\frak{so}(1,n)), && (n\geq3),\\
 &(\frak{sp}(2n,\bb{R}),\frak{sp}(n,\bb{C})), && (n\geq1), &&(\frak{e}_{6(-14)},\frak{f}_{4(-20)}).
\end{aligned}
\end{equation*}

\subsection{Holomorphic symmetric pairs} The explicit branching laws of $L^{\frak g}(c\zeta)|_{\frak{g}^\sigma}$ for holomorphic symmetric pairs $(\frak{g},\frak{g}^\sigma)$ with simple $\frak{g}$ are given in Section~\ref{sec:branchhol}. For the proofs we employ the following three methods:
\begin{itemize}
\item use Howe's dual pair correspondence and seesaw pairs (see Section~\ref{subsec:dualpair});
\item compute $K^\sigma$-types and identify irreducible constituents with Zuckerman's derived functor modules $A_{\frak q}(\lambda)$ (see Section~\ref{subsec:aqlambda});
\item use the Fock model as an explicit realization of the minimal holomorphic representation of $\frak{so}(2,n)$ (see Section~\ref{subsec:Fock}).
\end{itemize}
Each holomorphic pair is treated by at least one of these three methods (see Table~\ref{table:proof}).

We remark that for all holomorphic symmetric pairs $(\frak{g},\frak{g}^\sigma)$ the explicit branching law of $L^{\frak{g}}(x\zeta)|_{\frak{g}^\sigma}$ for large $x\gg0$ such that $L^{\frak{g}}(x\zeta)$ is a holomorphic discrete series representation was obtained by Kobayashi~\cite{Kob07}. For smaller values of $x$ the branching law is only known in some special cases, see e.g.\ \cite{Sek13} for the pair $(\frak{g},\frak{g}^\sigma)=(\frak{su}(n,n),\frak{so}^*(2n))$.

The tensor product of two representations can also be regarded as a branching problem for a holomorphic symmetric pair. In \cite{PZ}, the decomposition of the tensor product of two arbitrary highest weight modules of scalar type was computed. This yields in particular the decomposition of $L^{\frak{g}}(c\zeta)\otimes L^{\frak{g}}(c\zeta)$. 

\subsection{Non-holomorphic pairs} If $(\frak{g},\frak{g}^\sigma)$ is a non-holomorphic pair and $L^{\frak g}(c\zeta)|_{\frak{g}^\sigma}$ is discretely decomposable, then we prove in Section~\ref{sec:nonholirred} that the restriction stays irreducible.
In Section~\ref{sec:idnonhol} we further identify the restriction with a known representation, either Zuckerman's derived functor module $A_{\frak q}(\lambda)$, a complementary series representation or a small representation from \cite{HKM}.

\subsection{Kobayashi's conjecture} In Section~\ref{sec:AssVar}, we study a conjecture by Kobayashi on associated varieties in our particular setting (see Conjecture~\ref{conj:Kobayashi}). We prove that the conjecture is true in the following two cases:
\begin{itemize}
\item restriction of a highest (or lowest) weight module with respect to a holomorphic pair (not necessarily symmetric),
\item restriction of the minimal holomorphic representation with respect to non-holomorphic symmetric pairs assuming the discrete decomposability.
\end{itemize}
In particular, the conjecture is true for the branching laws in Sections~\ref{sec:branchhol} and \ref{sec:nonhol} as well as for tensor products of unitarizable highest (or lowest) weight modules.

\bigskip

{\bf Acknowledgments.}
The authors thank the American Institute of Mathematics for supporting the workshop \lq\lq{Branching Problems for Unitary Representations\rq\rq} in July 2011 at the MPI in Bonn where this work was started.

\section{Preliminaries}\label{sec:pre}

To describe the explicit branching formulas we fix some notation for each simple Lie algebra $\frak{g}$
 of Hermitian type.
We write the Dynkin diagram of $\frak{g}_\bb{C}$
 corresponding to the positive roots
 $\Delta^+(\frak{k}_\bb{C},\frak{t}_\bb{C})
 \cup\Delta(\frak{p}_+,\frak{t}_\bb{C})$
 and label the simple roots as $\alpha_i$.
The painted circle corresponds to the root in
 $\Delta(\frak{p}_+,\frak{t}_\bb{C})$.
We denote by $\omega_i\in\frak{t}_\bb{C}^*$
 the fundamental weight corresponding to $\alpha_i$
 and give the weight $c\zeta\in\sqrt{-1}\frak{t}^*$ in terms of $\omega_i$.

\begin{set}
\label{sumn}
Let $\frak{g}=\frak{su}(m,n)$.
We label the simple roots of $\frak{g}_\bb{C}$ as
\[\begin{xy}
\ar@{-} (0,0) *++!D{\alpha_1} *{\circ}="A";
  (10,0) *++!D{\alpha_2}  *{\circ}="B"
\ar@{-} "B"; (20,0)
\ar@{.} (20,0); (30,0) 
\ar@{-} (30,0); (40,0) *++!D{\alpha_{m-1}}  *{\circ}="C"
\ar@{-} "C"; (50,0) *++!D{\alpha_m}  *{\bullet}="D"
\ar@{-} "D"; (60,0) *++!D{\alpha_{m+1}}  *{\circ}="E"
\ar@{-} "E"; (70,0) 
\ar@{.} (70,0); (80,0) 
\ar@{-} (80,0); (90,0) *++!D{\alpha_{m+n-1}}  *{\circ}="F"
\end{xy}\]
Then we have $c\zeta=\omega_m$.
\end{set}

\begin{set}
\label{so2e}
Let $\frak{g}=\frak{so}(2,2n)$.
We label the simple roots of $\frak{g}_\bb{C}$ as
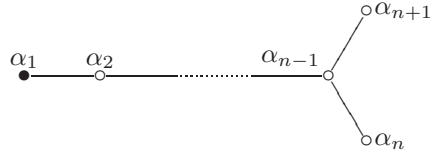
\[\begin{xy}
\ar@{-} (0,0) *+!D{\alpha_1} *{\bullet}="A"; 
 (10,0)  *+!D{\alpha_2} *{\circ}="B"
\ar@{-} "B";  (20,0)
\ar@{.} (20,0); (30,0) 
\ar@{-} (30,0); (40,0) *+!DR{\alpha_{n-1}} *{\circ}="F"
\ar@{-} "F"; (45,8.6)  *+!L{\alpha_{n+1}} *{\circ}
\ar@{-} "F"; (45,-8.6)  *+!L{\alpha_{n}} *{\circ}
\end{xy}\]
Then we have $c\zeta=(n-1)\omega_1$.
\end{set}

\begin{set}
\label{so2o}
Let $\frak{g}=\frak{so}(2,2n+1)$.
We label the simple roots of $\frak{g}_\bb{C}$ as
\[\begin{xy}
\ar@{-} (0,0) *+!D{\alpha_1} *{\bullet}="A"; 
 (10,0)  *+!D{\alpha_2} *{\circ}="B"
\ar@{-} "B";  (20,0)
\ar@{.} (20,0); (30,0) 
\ar@{-} (30,0); (40,0) *+!D{\alpha_n} *{\circ}="F"
\ar@{=>} "F"; (50,0)  *+!D{\alpha_{n+1}} *{\circ}
\end{xy}\]
Then we have $c\zeta=(n-\frac{1}{2})\omega_1$.
\end{set}

\begin{set}
\label{sostar}
Let $\frak{g}=\frak{so}^*(2n)$.
We label the simple roots of $\frak{g}_\bb{C}$ as
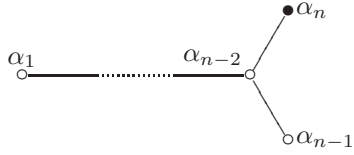
\[\begin{xy}
\ar@{-} (0,0) *+!D{\alpha_1} *{\circ}="A";  (10,0)
\ar@{.} (10,0); (20,0) 
\ar@{-} (20,0); (30,0) *+!DR{\alpha_{n-2}} *{\circ}="F"
\ar@{-} "F"; (35,8.6)  *+!L{\alpha_{n}} *{\bullet}
\ar@{-} "F"; (35,-8.6)  *+!L{\alpha_{n-1}} *{\circ}
\end{xy}\]
Then we have $c\zeta=2\omega_n$.
\end{set}

\begin{set}
\label{spr}
Let $\frak{g}=\frak{sp}(n,\bb{R})$.
We label the simple roots of $\frak{g}_\bb{C}$ as
\[\begin{xy}
\ar@{-} (0,0) *+!D{\alpha_1} *{\circ}="A";  (10,0)
\ar@{.} (10,0); (20,0) 
\ar@{-} (20,0); (30,0) *+!D{\alpha_{n-1}} *{\circ}="F"
\ar@{<=} "F"; (40,0) *+!D{\alpha_{n}} *{\bullet}
\end{xy}\]
Then we have $c\zeta=\frac{1}{2}\omega_n$.
\end{set}

\begin{set}
\label{e6(-14)}
Let $\frak{g}=\frak{e}_{6(-14)}(\equiv\frak{e}_6^3)$ 
so that $\frak{k}_\bb{C}=\frak{so}(10,\bb{C})\oplus\bb{C}$.
We label the simple roots of $\frak{g}_\bb{C}$ as
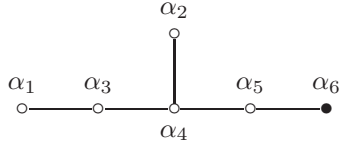
\[\begin{xy}
\ar@{-} (0,0) *++!D{\alpha_1} *{\circ}="A"; (10,0) *++!D{\alpha_3} 
 *{\circ}="B"
\ar@{-} "B"; (20,0)*++!U{\alpha_4}  *{\circ}="C"
\ar@{-} "C"; (30,0) *++!D{\alpha_5}  *{\circ}="D"
\ar@{-} "C"; (20,10) *++!D{\alpha_2}  *{\circ}="G"
\ar@{-} "D"; (40,0) *++!D{\alpha_6}  *{\bullet}="E"
\end{xy}\]
Then we have $c\zeta=3\omega_6$.
\end{set}

\begin{set}
\label{e7(-25)}
Let $\frak{g}=\frak{e}_{7(-25)}(\equiv\frak{e}_7^3)$ 
so that $\frak{k}_\bb{C}=\frak{e}_{6,\bb{C}}\oplus\bb{C}$.
We label the simple roots of $\frak{g}_\bb{C}$ as
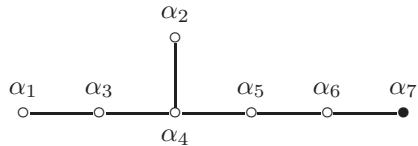
\[\begin{xy}
\ar@{-} (0,0) *++!D{\alpha_1} *{\circ}="A"; (10,0) *++!D{\alpha_3} 
 *{\circ}="B"
\ar@{-} "B"; (20,0)*++!U{\alpha_4}  *{\circ}="C"
\ar@{-} "C"; (30,0) *++!D{\alpha_5}  *{\circ}="D"
\ar@{-} "C"; (20,10) *++!D{\alpha_2}  *{\circ}="G"
\ar@{-} "D"; (40,0) *++!D{\alpha_6}  *{\circ}="E"
\ar@{-} "E"; (50,0) *++!D{\alpha_7}  *{\bullet}="F"
\end{xy}\]
Then we have $c\zeta=4\omega_7$.
\end{set}

\section{Branching laws for holomorphic symmetric pairs}\label{sec:branchhol}

Let $(\frak{g},\frak{g}^\sigma)$ be a symmetric pair of holomorphic type.
By \cite[Fact 5.4]{Kob98ii}, 
 any unitarizable lowest weight module of $\frak{g}$ is decomposed
 as a direct sum of lowest weight modules of $\frak{g}^\sigma$.
 
\begin{thm}
The explicit branching rules for the restriction of the minimal holomorphic representation $L^\frak{g}(c\zeta)$ of $\frak{g}$ to any symmetric subalgebra $\frak{g}^\sigma$ of holomorphic type are given by the formulas \eqref{eq:umnpq} -- \eqref{eq:e7sostar}.
\end{thm}

\subsection{$(\frak{g},\frak{g}^\sigma)=(\frak{su}(m,n),\frak{su}(p,q)\oplus\frak{su}(m-p,n-q)
 \oplus\frak{u}(1))$}

Let $\alpha_i$ be as in Setting~\ref{sumn}.
We may assume that $\sigma=1$ on $\frak{t}$,
 $\sigma=1$ on $\frak{g}_{\alpha_i}$ for $i\neq p,m+n-q$ and 
 $\sigma=-1$ on $\frak{g}_{\alpha_i}$ for $i=p,m+n-q$.
We put
\begin{align*}
&\beta_i:=\alpha_i \ (1\leq i\leq p-1), \quad
\beta_p:=\alpha_p+\alpha_{p+1}+\cdots+\alpha_{m+n-q},\\
&\beta_i:= \alpha_{i+m+n-p-q} \ (p+1\leq i\leq p+q-1),\\
&\gamma_i:=\alpha_{p+i}\ (1\leq i\leq m+n-p-q-1).
\end{align*}
Then $\beta_i$ form a set of simple roots for $\frak{sl}(p+q,\bb{C})$,
 the first component of $\frak{g}^\sigma_\bb{C}$, and
 $\gamma_i$ form one for $\frak{sl}(m-p+n-q,\bb{C})$, the second component.
Write $\mu_i$ and $\nu_i$ for the fundamental weights corresponding to
 $\beta_i$ and $\gamma_i$, respectively.
Let $e\in\frak{t}$ be the vector in the $\frak{u}(1)$-component
 of $\frak{g}^\sigma$ such that $\alpha_p(e)=\sqrt{-1}$.
Then the restriction $L^{\frak{g}}(\omega_m)|_{\frak{g}^\sigma}$ decomposes as 
\begin{align}
\label{eq:umnpq}
L^{\frak{g}}(\omega_m)|_{\frak{g}^\sigma} \simeq{}& \bigoplus_{k=0}^{\infty}
L^{\frak{su}(p,q)}(\mu_p+k\mu_{p+q-1})\boxtimes
L^{\frak{su}(m-p,n-q)}(k\nu_1+\nu_{m-p})\boxtimes
\bb{C}_{-k+\frac{np-mq}{m+n}}\\ 
& \oplus\bigoplus_{k=1}^{\infty}
L^{\frak{su}(p,q)}(k\mu_1+\mu_p)\boxtimes
L^{\frak{su}(m-p,n-q)}(\nu_{m-p}+k\nu_{m+n-p-q-1})\boxtimes
\bb{C}_{k+\frac{np-mq}{m+n}}\nonumber
\end{align}
if $p,q,m-p,n-q\geq 1$
and 
\begin{align}
\label{eq:umnpn}
L^{\frak{g}}(\omega_m)|_{\frak{g}^\sigma}
\simeq
\bigoplus_{k=0}^{\infty}
L^{\frak{su}(p,q)}(\mu_p+k\mu_{p+q-1})\boxtimes
F^{\frak{su}(m-p)}(k\nu_1)\boxtimes
\bb{C}_{-k+\frac{n(p-m)}{m+n}}
\end{align}
if $n=q$ and $p,q,m-p\geq 1$.
Here $\bb{C}_a$ is the character of $\frak{u}(1)$-component
 of $\frak{g}^\sigma$ on which $e$ acts as $\sqrt{-1}a$.

\smallskip
\subsection{$(\frak{g},\frak{g}^\sigma)=(\frak{su}(n,n),\frak{so}^*(2n))$}

Let $\alpha_i$ be as in Setting~\ref{sumn} for $m=n$.
We may assume $\sigma\frak{t}=\frak{t}$,
 $\sigma\alpha_i=\alpha_{2n-i}$ for $1\leq i\leq n$,
 and $\sigma=-1$ on $\frak{g}_{\alpha_n}$.
Then $\frak{t}^\sigma$ is a Cartan subalgebra of $\frak{k}^\sigma$.
We put
\begin{align*}
\beta_i:=\alpha_i|_{\frak{t}^\sigma}\ (1\leq i\leq n-1),\quad
\beta_n:=(\alpha_{n-1}+\alpha_n)|_{\frak{t}^\sigma}.
\end{align*}
Then $\beta_i$ form a set of simple roots for $\frak{g}^\sigma_\bb{C}$.
Write $\mu_i$ for the corresponding fundamental weights.
Then the restriction $L^{\frak{g}}(\omega_n)|_{\frak{g}^\sigma}$ decomposes as 
\begin{align}\label{eq:unnsostar}
L^{\frak{g}}(\omega_n)|_{\frak{g}^\sigma}
\simeq
\bigoplus_{k=0}^{\infty} L^{\frak{g}^\sigma}(2k\mu_1+2\mu_n).
\end{align}

\smallskip
\subsection{$(\frak{g},\frak{g}^\sigma)=(\frak{su}(n,n),\frak{sp}(n,\bb{R}))$}

Let $\alpha_i$ be as in Setting~\ref{sumn} for $m=n$.
We may assume $\sigma\frak{t}=\frak{t}$,
 $\sigma\alpha_i=\alpha_{2n-i}$ for $1\leq i\leq n$,
 and $\sigma=1$ on $\frak{g}_{\alpha_n}$.
Then $\frak{t}^\sigma$ is a Cartan subalgebra of $\frak{k}^\sigma$.
We put
\begin{align*}
\beta_i:=\alpha_i|_{\frak{t}^\sigma}\ (1\leq i\leq n).
\end{align*}
Then $\beta_i$ form a set of simple roots for $\frak{g}^\sigma_\bb{C}$.
Write $\mu_i$ for the corresponding fundamental weights.
Then the restriction $L^{\frak{g}}(\omega_n)|_{\frak{g}^\sigma}$ decomposes as 
\begin{align}\label{eq:unnspn}
L^{\frak{g}}(\omega_n)|_{\frak{g}^\sigma}
\simeq L^{\frak{g}^\sigma}(\mu_n)\oplus L^{\frak{g}^\sigma}(\mu_2+\mu_n).
\end{align}

\smallskip
\subsection{$(\frak{g},\frak{g}^\sigma)=(\frak{so}(2,2n),\frak{su}(1,n)\oplus\frak{u}(1))$}

Let $\alpha_i$ be as in Setting~\ref{so2e}.
We may assume $\sigma=1$ on $\frak{t}$, 
 $\sigma=1$ on $\frak{g}_{\alpha_i}$ for $1\leq i\leq n$,
 and $\sigma=-1$ on $\frak{g}_{\alpha_{n+1}}$.
We put 
\begin{align*}
\beta_i:=\alpha_i\ (1\leq i\leq n).
\end{align*}
Then $\beta_i$ form a set of simple roots for $\frak{sl}(n+1,\bb{C})$.
Write $\mu_i$ for the corresponding fundamental weights.
Let $e\in\frak{t}$ be the vector in the $\frak{u}(1)$-component
 of $\frak{g}^\sigma$ such that $\alpha_{n+1}(e)=\sqrt{-1}$.
Then the restriction $L^{\frak{g}}((n-1)\omega_1)|_{\frak{g}^\sigma}$
 decomposes as 
\begin{align}
\label{eq:sou}
L^{\frak{g}}((n-1)\omega_1)|_{\frak{g}^\sigma}
\simeq
\bigoplus_{k=0}^{\infty} L^{\frak{su}(1,n)}((n-1)\mu_1+k\mu_2)
\boxtimes \bb{C}_{k+\frac{n-1}{2}},
\end{align}
where $\bb{C}_a$ is the character of $\frak{u}(1)$-component
 of $\frak{g}^\sigma$ on which $e$ acts as $\sqrt{-1}a$.

\smallskip
\subsection{$(\frak{g},\frak{g}^\sigma)=(\frak{so}(2,2n),\frak{so}(2,m)\oplus\frak{so}(2n-m))$}

Let $\alpha_i$ be as in Setting~\ref{so2e}.

Suppose first that $m$ is even and $m=2l$.
Then we may assume $\sigma=1$ on $\frak{t}$.
If $n-l>1$, suppose $\sigma=1$ on $\frak{g}_{\alpha_i}$ for $i\neq l+1$
 and $\sigma=-1$ on $\frak{g}_{\alpha_{l+1}}$.
If $n-l=1$, suppose $\sigma=1$ on $\frak{g}_{\alpha_i}$ for $i<n$
 and $\sigma=-1$ on $\frak{g}_{\alpha_i}$ for $i=n, n+1$.
We put 
\begin{align*}
&\beta_i:=\alpha_i\ (1\leq i\leq l), \quad
\beta_{l+1}:=\alpha_l+2\alpha_{l+1}+\cdots
 +2\alpha_{n-1}+\alpha_n+\alpha_{n+1},\\
&\gamma_i:=\alpha_{i+l+1}\ (1\leq i\leq n-l).
\end{align*}
Then $\beta_i$ form a set of simple roots for $\frak{so}(2l+2,\bb{C})$,
 the first component of $\frak{g}^\sigma_\bb{C}$
 and $\gamma_i$ form one for $\frak{so}(2n-2l,\bb{C})$ if $n-l\geq 2$.
Write $\mu_i$ and $\nu_i$ for the fundamental weights corresponding to
 $\beta_i$ and $\gamma_i$, respectively.
If $n-l=1$,
 let $e\in\frak{t}$ be the vector in the $\frak{so}(2n-2l)$-component
 of $\frak{g}^\sigma$ such that $\alpha_{n+1}(e)=\sqrt{-1}$.
Then the restriction $L^{\frak{g}}((n-1)\omega_1)|_{\frak{g}^\sigma}$
 decomposes as 
\begin{align}\label{eq:so2eso2e}
L^{\frak{g}}((n-1)\omega_1)|_{\frak{g}^\sigma}
\simeq
\bigoplus_{k=0}^{\infty} L^{\frak{so}(2,2l)}((n+k-1)\mu_1)
\boxtimes F^{\frak{so}(2n-2l)}(k\nu_1)
\end{align}
if $n-l\geq 2$ and 
\begin{align}
L^{\frak{g}}((n-1)\omega_1)|_{\frak{g}^\sigma}
\simeq
\bigoplus_{k=-\infty}^{\infty} L^{\frak{so}(2,2n-2)}((n+|k|-1)\mu_1)
\boxtimes \bb{C}_k
\end{align}
if $n-l=1$,
where $\bb{C}_a$ is the character of $\frak{so}(2)$-component
 of $\frak{g}^\sigma$ on which $e$ acts as $\sqrt{-1}a$.

Suppose next that $m$ is odd and $m=2l+1$.
Then we may assume $\sigma\frak{t}=\frak{t}$,
 $\sigma\alpha_i=\alpha_i$ for $1\leq i\leq n-1$, 
 and $\sigma\alpha_n=\alpha_{n+1}$.
If $n-l>1$, suppose
 $\sigma=1$ on $\frak{g}_{\alpha_i}$ for $1\leq i\leq l$ or
 $l+2\leq i\leq n-1$,
 and $\sigma=-1$ on $\frak{g}_{\alpha_{l+1}}$.
If $n-l=1$, suppose 
 $\sigma=1$ on $\frak{g}_{\alpha_i}$ for $1\leq i\leq n-1$.
Then $\frak{t}^\sigma$ is a Cartan subalgebra of $\frak{k}^\sigma$.
We put 
\begin{align*}
&\beta_i:=\alpha_i|_{\frak{t}^\sigma}\ (1\leq i\leq l), \quad
\beta_{l+1}:=(\alpha_{l+1}+\alpha_{l+2}+\cdots+\alpha_n)|_{\frak{t}^\sigma},\\
&\gamma_i:=\alpha_{i+l+1}|_{\frak{t}^\sigma}\ (1\leq i\leq n-l-1).
\end{align*}
Then $\beta_i$ form a set of simple roots for $\frak{so}(2l+3,\bb{C})$
 and $\gamma_i$ form one for $\frak{so}(2n-2l-1,\bb{C})$.
Write $\mu_i$ and $\nu_i$ for the fundamental weights corresponding to
 $\beta_i$ and $\gamma_i$, respectively.
Then the restriction $L^{\frak{g}}((n-1)\omega_1)|_{\frak{g}^\sigma}$
 decomposes as 
\begin{align}
L^{\frak{g}}((n-1)\omega_1)|_{\frak{g}^\sigma}
\simeq
\bigoplus_{k=0}^{\infty} L^{\frak{so}(2,2l+1)}((n+k-1)\mu_1)
\boxtimes F^{\frak{so}(2n-2l-1)}(k\nu_1)
\end{align}
if $n-l\geq 2$ and
\begin{align}\label{eq:so2eso2e-1}
L^{\frak{g}}((n-1)\omega_1)|_{\frak{g}^\sigma}
\simeq
L^{\frak{g}^\sigma}((n-1)\mu_1)\oplus L^{\frak{g}^\sigma}(n\mu_1)
\end{align}
if $n-l=1$.

\smallskip
\subsection{$(\frak{g},\frak{g}^\sigma)=(\frak{so}(2,2n+1),\frak{so}(2,m)\oplus\frak{so}(2n-m+1))$}

Let $\alpha_i$ be as in Setting~\ref{so2o}.

Suppose first that $m$ is even and $m=2l$.
Then we may assume $\sigma=1$ on $\frak{t}$,
 $\sigma=1$ on $\frak{g}_{\alpha_i}$ for $i\neq l+1$,
 and $\sigma=-1$ on $\frak{g}_{\alpha_{l+1}}$.
We put 
\begin{align*}
&\beta_i:=\alpha_i\ (1\leq i\leq l), \quad
\beta_{l+1}:=\alpha_l+2\alpha_{l+1}+\cdots+2\alpha_{n+1},\\
&\gamma_i:=\alpha_{i+l+1}\ (1\leq i\leq n-l).
\end{align*}
Then $\beta_i$ form a set of simple roots for $\frak{so}(2l+2,\bb{C})$
 and $\gamma_i$ form one for $\frak{so}(2n-2l+1,\bb{C})$.
Write $\mu_i$ and $\nu_i$ for the fundamental weights corresponding to
 $\beta_i$ and $\gamma_i$, respectively.
Then the restriction $L^{\frak{g}}((n-\frac{1}{2})\omega_1)|_{\frak{g}^\sigma}$
 decomposes as 
\begin{align}
L^{\frak{g}}\Bigl(\Bigl(n-\frac{1}{2}\Bigr)\omega_1\Bigr)
\Bigl|_{\frak{g}^\sigma}
\simeq
\bigoplus_{k=0}^{\infty} L^{\frak{so}(2,2l)}
 \Bigl(\Bigl(n+k-\frac{1}{2}\Bigr)\mu_1\Bigr)
\boxtimes F^{\frak{so}(2n-2l+1)}(k\nu_1)
\end{align}
if $n>l$ and 
\begin{align}\label{eq:so2oso2o-1}
L^{\frak{g}}\Bigl(\Bigl(n-\frac{1}{2}\Bigr)\omega_1\Bigr)
\Bigl|_{\frak{g}^\sigma}
\simeq
 L^{\frak{g}^\sigma}
 \Bigl(\Bigl(n-\frac{1}{2}\Bigr)\mu_1\Bigr)\oplus
 L^{\frak{g}^\sigma}
 \Bigl(\Bigl(n+\frac{1}{2}\Bigr)\mu_1\Bigr)
\end{align}
if $n=l$.

Suppose next that $m$ is odd and $m=2l+1$.
Then we may assume $\sigma=1$ on $\frak{t}$,
 $\sigma=1$ on $\frak{g}_{\alpha_i}$ for $i\neq l+1, n+1$,
 and $\sigma=-1$ on $\frak{g}_{\alpha_i}$ for $i=l+1, n+1$.
We put 
\begin{align*}
&\beta_i:=\alpha_i\ (1\leq i\leq l), \quad
\beta_{l+1}:=\alpha_{l+1}+\cdots+\alpha_{n+1}, \\
&\gamma_i:=\alpha_{i+l+1}\ (1\leq i\leq n-l-1), \quad
\gamma_{n-l}:=\alpha_{n}+2\alpha_{n+1}.
\end{align*}
Then $\beta_i$ form a set of simple roots for $\frak{so}(2l+3,\bb{C})$
 and $\gamma_i$ form one for $\frak{so}(2n-2l,\bb{C})$ if $n-l\geq 2$.
Write $\mu_i$ and $\nu_i$ for the fundamental weights corresponding to
 $\beta_i$ and $\gamma_i$, respectively.
If $n-l=1$,
 let $e\in\frak{t}$ be the vector in the $\frak{so}(2n-2l)$-component
 of $\frak{g}^\sigma$ such that $\alpha_{n+1}(e)=\sqrt{-1}$.
Then the restriction $L^{\frak{g}}((n-\frac{1}{2})\omega_1)|_{\frak{g}^\sigma}$
 decomposes as 
\begin{align}
L^{\frak{g}}\Bigl(\Bigl(n-\frac{1}{2}\Bigr)\omega_1\Bigr)
\Bigl|_{\frak{g}^\sigma}
\simeq
\bigoplus_{k=0}^{\infty} L^{\frak{so}(2,2l+1)}
\Bigl(\Bigl(n+k-\frac{1}{2}\Bigr)\mu_1\Bigr)
\boxtimes F^{\frak{so}(2n-2l)}(k\nu_1)
\end{align}
if $n-l\geq 2$ and 
\begin{align}\label{eq:so2oso2o-2}
L^{\frak{g}}\Bigl(\Bigl(n-\frac{1}{2}\Bigr)\omega_1\Bigr)
\Bigl|_{\frak{g}^\sigma}
\simeq
\bigoplus_{k=-\infty}^{\infty} L^{\frak{so}(2,2n-1)}
\Bigl(\Bigl(n+|k|-\frac{1}{2}\Bigr)\mu_1\Bigr)
\boxtimes \bb{C}_k
\end{align}
if $n-l=1$,
where $\bb{C}_a$ is the character of $\frak{so}(2)$-component
 of $\frak{g}^\sigma$ on which $e$ acts as $\sqrt{-1}a$.

\smallskip
\subsection{$(\frak{g},\frak{g}^\sigma)=(\frak{so}^*(2n),\frak{su}(m,n-m)\oplus\frak{u}(1))$}

Let $\alpha_i$ be as in Setting~\ref{sostar}.
We may assume $\sigma=1$ on $\frak{t}$.
If $m<n-1$, suppose $\sigma=1$ on $\frak{g}_{\alpha_i}$ for $i\neq m,n$
 and $\sigma=-1$ on $\frak{g}_{\alpha_i}$ for $i=m,n$.
If $m=n-1$, suppose $\sigma=1$ on $\frak{g}_{\alpha_i}$ for $i\neq n-1$
 and $\sigma=-1$ on $\frak{g}_{\alpha_{n-1}}$.
We put 
\begin{align*}
&\beta_i:=\alpha_i\ (1\leq i\leq m-1), \quad
\beta_m:=\alpha_m+\cdots+\alpha_{n-2}+\alpha_n, \\
&\beta_i:=\alpha_{m+n-i}\ (m+1\leq i\leq n-1),
\end{align*}
if $m<n-1$ and
\begin{align*}
\beta_i:=\alpha_i\ (1\leq i\leq n-2), \quad
\beta_{n-1}:=\alpha_{n-2}+\alpha_{n-1}+\alpha_n,
\end{align*}
if $m=n-1$.
Then $\beta_i$ form a set of simple roots for $\frak{sl}(n,\bb{C})$.
Write $\mu_i$ for the corresponding fundamental weights.
Let $e\in\frak{t}$ be the vector in the $\frak{u}(1)$-component
 of $\frak{g}^\sigma$ such that $\alpha_m(e)=\sqrt{-1}$.
Then the restriction $L^{\frak{g}}(2\omega_n)|_{\frak{g}^\sigma}$
 decomposes as 
\begin{align}\label{eq:sostaru}
L^{\frak{g}}(2\omega_n)|_{\frak{g}^\sigma}
\simeq{}&\bigoplus_{k=0}^{\infty}
L^{\frak{su}(m,n-m)}(2\mu_m+k\mu_{n-2})\boxtimes \bb{C}_{-k-\frac{n}{2}+m}
\\ \nonumber
&\oplus\bigoplus_{k=1}^{\infty}
L^{\frak{su}(m,n-m)}(k\mu_2+2\mu_m)\boxtimes \bb{C}_{k-\frac{n}{2}+m}
\end{align}
if $2\leq m\leq n-2$, 
\begin{align}\label{eq:sostaru1n-1}
L^{\frak{g}}(2\omega_n)|_{\frak{g}^\sigma}
\simeq\bigoplus_{k=0}^{\infty}
L^{\frak{su}(1,n-1)}(2\mu_1+k\mu_{n-2})\boxtimes \bb{C}_{-k-\frac{n}{2}+1}
\end{align}
if $m=1$, and
\begin{align}\label{eq:sostarun-11}
L^{\frak{g}}(2\omega_n)|_{\frak{g}^\sigma}
\simeq\bigoplus_{k=0}^{\infty}
L^{\frak{su}(n-1,1)}(k\mu_2+2\mu_{n-1})\boxtimes \bb{C}_{k+\frac{n}{2}-1}
\end{align}
if $m=n-1$.
Here, $\bb{C}_a$ is the character of $\frak{u}(1)$-component
 of $\frak{g}^\sigma$ on which $e$ acts as $\sqrt{-1}a$.

\smallskip
\subsection{$(\frak{g},\frak{g}^\sigma)=(\frak{so}^*(2n),\frak{so}^*(2m)\oplus\frak{so}^*(2n-2m))$}

Let $\alpha_i$ be as in Setting~\ref{sostar}.
We may assume $\sigma=1$ on $\frak{t}$.
If $n-m>1$, suppose $\sigma=1$ on $\frak{g}_{\alpha_i}$ for $i\neq m$
 and $\sigma=-1$ on $\frak{g}_{\alpha_m}$ and 
 we put 
\begin{align*}
&\beta_i:=\alpha_i\ (1\leq i\leq m-1), \quad
\beta_m:=\alpha_m+2\alpha_{m+1}+\cdots+2\alpha_{n-2}+\alpha_{n-1}+\alpha_n,\\
&\gamma_i:=\alpha_{i+m}\ (1\leq i\leq n-m).
\end{align*}
If $n-m=1$, suppose $\sigma=1$ on $\frak{g}_{\alpha_i}$ for $i<n-1$
 and $\sigma=-1$ on $\frak{g}_{\alpha_i}$ for $i=n-1,n$ and we put
\begin{align*}
\beta_i:=\alpha_i\ (1\leq i\leq n-2), \quad
\beta_{n-1}:=\alpha_{n-2}+\alpha_{n-1}+\alpha_n.
\end{align*}
Then $\beta_i$ form a set of simple roots for $\frak{so}(2m,\bb{C})$
 if $m>1$ and $\gamma_i$ form one for $\frak{so}(2n-2m,\bb{C})$ if $n-m>1$.
Write $\mu_i$ and $\nu_i$ for the fundamental weights corresponding to
 $\beta_i$ and $\gamma_i$, respectively.
If $m=1$,
 let $e\in\frak{t}$ be the vector in the $\frak{so}^*(2m)$-component
 of $\frak{g}^\sigma$ such that $\alpha_1(e)=\sqrt{-1}$.
If $n-m=1$,
 let $e\in\frak{t}$ be the vector in the $\frak{so}^*(2n-2m)$-component
 of $\frak{g}^\sigma$ such that $\alpha_n(e)=\sqrt{-1}$.
Then the restriction $L^{\frak{g}}(2\omega_n)|_{\frak{g}^\sigma}$ 
decomposes as 
\begin{align}\label{eq:sostarsostar}
L^{\frak{g}}(2\omega_n)|_{\frak{g}^\sigma}
\simeq
\bigoplus_{k=0}^{\infty} L^{\frak{so}^*(2m)}(k\mu_1+2\mu_m)
\boxtimes  L^{\frak{so}^*(2n-2m)}(k\nu_1+2\nu_{n-m})
\end{align}
if $2\leq m \leq n-2$,
\begin{align}\label{eq:sostarsostar1n-1}
L^{\frak{g}}(2\omega_n)|_{\frak{g}^\sigma}
\simeq
\bigoplus_{k=0}^{\infty} \bb{C}_{k+1}
\boxtimes  L^{\frak{so}^*(2n-2)}(k\nu_1+2\nu_{n-1})
\end{align}
if $m=1$, and 
\begin{align}\label{eq:sostarsostarn-11}
L^{\frak{g}}(2\omega_n)|_{\frak{g}^\sigma}
\simeq
\bigoplus_{k=0}^{\infty} 
  L^{\frak{so}^*(2n-2)}(k\mu_1+2\mu_{n-1}) \boxtimes \bb{C}_{k+1}
\end{align}
if $n-m=1$.
Here, $\bb{C}_a$ is the character of $\frak{u}(1)$-component
 of $\frak{g}^\sigma$ on which $e$ acts as $\sqrt{-1}a$.

\smallskip
\subsection{$(\frak{g},\frak{g}^\sigma)=(\frak{sp}(n,\bb{R}),\frak{su}(m,n-m)\oplus\frak{u}(1))$}

Let $\alpha_i$ be as in Setting~\ref{spr}.
We may assume $\sigma=1$ on $\frak{t}$,
 $\sigma=1$ on $\frak{g}_{\alpha_i}$ for $i\neq m,n$,
 and $\sigma=-1$ on $\frak{g}_{\alpha_i}$ for $i=m,n$.
We put 
\begin{align*}
&\beta_i:=\alpha_i\ (1\leq i\leq m-1), \quad
\beta_m:=\alpha_m+\cdots+\alpha_n, \\
&\beta_i:= \alpha_{m+n-i}\ (m+1\leq i\leq n-1).
\end{align*}
Then $\beta_i$ form a set of simple roots for $\frak{sl}(n,\bb{C})$.
Write $\mu_i$ for the corresponding fundamental weights.
Let $e\in\frak{t}$ be the vector in the $\frak{u}(1)$-component
 of $\frak{g}^\sigma$ such that $\alpha_m(e)=\sqrt{-1}$.
Then the restriction $L^{\frak{g}}(\frac{1}{2}\omega_n)|_{\frak{g}^\sigma}$
 decomposes as 
\begin{align}\label{eq:spu}
L^{\frak{g}}\Bigl(\frac{1}{2}\omega_n\Bigr)\Bigl|_{\frak{g}^\sigma}
\simeq{}&\bigoplus_{k=0}^{\infty}
L^{\frak{su}(m,n-m)}(\mu_m+2k\mu_{n-1})
 \boxtimes \bb{C}_{-k-\frac{n}{4}+\frac{m}{2}}\\ \nonumber
&\oplus\bigoplus_{k=1}^{\infty}
 L^{\frak{su}(m,n-m)}(2k\mu_1+\mu_m)
 \boxtimes \bb{C}_{k-\frac{n}{4}+\frac{m}{2}},
\end{align}
 where $\bb{C}_a$ is the character of $\frak{u}(1)$-component
 of $\frak{g}^\sigma$ on which $e$ acts as $\sqrt{-1}a$.

\smallskip
\subsection{$(\frak{g},\frak{g}^\sigma)=(\frak{sp}(n,\bb{R}),\frak{sp}(m,\bb{R})\oplus\frak{sp}(n-m,\bb{R}))$}

Let $\alpha_i$ be as in Setting~\ref{spr}.
We may assume $\sigma=1$ on $\frak{t}$,
 $\sigma=1$ on $\frak{g}_{\alpha_i}$ for $i\neq m$,
 and $\sigma=-1$ on $\frak{g}_{\alpha_m}$.
We put 
\begin{align*}
&\beta_i:=\alpha_i\ (1\leq i\leq m-1), \quad
\beta_m:=2\alpha_m+2\alpha_{m+1}+\cdots+2\alpha_{n-1}+\alpha_n,\\
&\gamma_i:=\alpha_{i+m}\ (1\leq i\leq n-m).
\end{align*}
Then $\beta_i$ form a set of simple roots for $\frak{sp}(m,\bb{C})$
 and $\gamma_i$ form one for $\frak{sp}(n-m,\bb{C})$.
Write $\mu_i$ and $\nu_i$ for the fundamental weights corresponding to
 $\beta_i$ and $\gamma_i$, respectively.
Then the restriction $L^{\frak{g}}(\frac{1}{2}\omega_n)|_{\frak{g}^\sigma}$
 decomposes as 
\begin{align}
L^{\frak{g}}\Bigl(\frac{1}{2}\omega_n\Bigr)\Bigl|_{\frak{g}^\sigma}
\simeq{}&
 \Bigl(L^{\frak{sp}(m,\bb{R})}\Bigl(\frac{1}{2}\mu_m\Bigr)
\boxtimes  L^{\frak{sp}(n-m,\bb{R})}\Bigl(\frac{1}{2}\nu_{n-m}\Bigr)\Bigr) \\ \nonumber
&\oplus\Bigl(L^{\frak{sp}(m,\bb{R})}\Bigl(\mu_1+\frac{1}{2}\mu_m\Bigr)
\boxtimes  L^{\frak{sp}(n-m,\bb{R})}\Bigl(\nu_1+\frac{1}{2}\nu_{n-m}\Bigr)\Bigr).
\end{align}

\smallskip
\subsection{$(\frak{g},\frak{g}^\sigma)=(\frak{e}_{6(-14)},\frak{so}(2,8)\oplus\frak{so}(2))$}

Let $\alpha_i$ be as in Setting~\ref{e6(-14)}.
We may assume $\sigma=1$ on $\frak{t}$,
 $\sigma=1$ on $\frak{g}_{\alpha_i}$ for $i\neq 1$,
 and $\sigma=-1$ on $\frak{g}_{\alpha_1}$.
We put 
\begin{align*}
\beta_i:=\alpha_{7-i}\ (1\leq i\leq 5).
\end{align*}
Then $\beta_i$ form a set of simple roots for $\frak{so}(10,\bb{C})$,
 the first component of $\frak{g}^\sigma_\bb{C}$.
Write $\mu_i$ for the corresponding fundamental weights.
Let $e\in\frak{t}$ be the vector in the $\frak{u}(1)$-component
 of $\frak{g}^\sigma$ such that $\alpha_1(e)=\sqrt{-1}$.
Then the restriction $L^{\frak{g}}(3\omega_6)|_{\frak{g}^\sigma}$
 decomposes as 
\begin{align}\label{eq:e6so}
L^{\frak{g}}(3\omega_6)|_{\frak{g}^\sigma}
&\simeq\bigoplus_{k=0}^{\infty}
L^{\frak{so}(2,8)}(3\mu_1+k\mu_5)\boxtimes \bb{C}_{k+2},
\end{align}
 where $\bb{C}_a$ is the character of $\frak{so}(2)$-component
 of $\frak{g}^\sigma$ on which $e$ acts as $\sqrt{-1}a$.

\smallskip
\subsection{$(\frak{g},\frak{g}^\sigma)=(\frak{e}_{6(-14)},\frak{su}(4,2)\oplus\frak{su}(2))$}

Let $\alpha_i$ be as in Setting~\ref{e6(-14)}.
We may assume $\sigma=1$ on $\frak{t}$,
 $\sigma=1$ on $\frak{g}_{\alpha_i}$ for $i\neq 3$,
 and $\sigma=-1$ on $\frak{g}_{\alpha_3}$.
We put
\begin{align*}
&\beta_1:=\alpha_2,\quad \beta_2:=\alpha_4,\quad \beta_3=\alpha_5,\quad
\beta_4:=\alpha_6,\\
&\beta_5:=\alpha_1+\alpha_2+2\alpha_3+2\alpha_4+\alpha_5,\quad
 \gamma_1:=\alpha_1.
\end{align*}
Then $\beta_i$ form a set of simple roots for $\frak{sl}(6,\bb{C})$
 and $\gamma_1$ is a root for $\frak{sl}(2,\bb{C})$.
Write $\mu_i$ and $\nu_1$ for the fundamental weights corresponding to
 $\beta_i$ and $\gamma_1$, respectively.
Then the restriction $L^{\frak{g}}(3\omega_6)|_{\frak{g}^\sigma}$
 decomposes as 
\begin{align}
L^{\frak{g}}(3\omega_6)|_{\frak{g}^\sigma}
&\simeq\bigoplus_{k=0}^{\infty}
L^{\frak{su}(4,2)}(k\mu_3+3\mu_4)\boxtimes
F^{\frak{su}(2)}(k\nu_1).
\end{align}

\smallskip
\subsection{$(\frak{g},\frak{g}^\sigma)=(\frak{e}_{6(-14)},\frak{so}^*(10)\oplus\frak{so}(2))$}

Let $\alpha_i$ be as in Setting~\ref{e6(-14)}.
We may assume $\sigma=1$ on $\frak{t}$,
 $\sigma=1$ on $\frak{g}_{\alpha_i}$ for $i\neq 2,6$,
 and $\sigma=-1$ on $\frak{g}_{\alpha_i}$ for $i=2,6$.
We put
\begin{align*}  \beta_i:=\alpha_{6-i}\ (1\leq i\leq 3),\quad
 \beta_4:=\alpha_1,\quad \beta_5:=\alpha_2+\alpha_4+\alpha_5+\alpha_6.
\end{align*}
Then $\beta_i$ form a set of simple roots for $\frak{so}(10,\bb{C})$.
Write $\mu_i$ for the corresponding fundamental weights.
Let $e\in\frak{t}$ be the vector in the $\frak{u}(1)$-component
 of $\frak{g}^\sigma$ such that $\alpha_2(e)=\sqrt{-1}$.
Then the restriction $L^{\frak{g}}(3\omega_6)|_{\frak{g}^\sigma}$
 decomposes as 
\begin{align}
L^{\frak{g}}(3\omega_6)|_{\frak{g}^\sigma}
\simeq\bigoplus_{k=0}^{\infty}
\Bigl(L^{\frak{so}^*(10)}((k+3)\mu_5)\boxtimes \bb{C}_{k-1}\Bigr)
\oplus\bigoplus_{k=1}^{\infty}
\Bigl(L^{\frak{so}^*(10)}(k\mu_4+3\mu_5)\boxtimes \bb{C}_{-k-1}\Bigr),
\end{align}
 where $\bb{C}_a$ is the character of $\frak{so}(2)$-component
 of $\frak{g}^\sigma$ on which $e$ acts as $\sqrt{-1}a$.

\smallskip
\subsection{$(\frak{g},\frak{g}^\sigma)=(\frak{e}_{6(-14)},
\frak{su}(5,1)\oplus\frak{sp}(1,\bb{R}))$}

Let $\alpha_i$ be as in Setting~\ref{e6(-14)}.
We may assume $\sigma=1$ on $\frak{t}$,
 $\sigma=1$ on $\frak{g}_{\alpha_i}$ for $i\neq 2$,
 and $\sigma=-1$ on $\frak{g}_{\alpha_2}$.
We put
\begin{align*}
\beta_1:=\alpha_1,\quad \beta_i:=\alpha_{i+1}\ (2\leq i\leq 5),\quad
\gamma_1:=\alpha_1+2\alpha_2+2\alpha_3+3\alpha_4+2\alpha_5+\alpha_6.
\end{align*}
Then $\beta_i$ form a set of simple roots for $\frak{sl}(6,\bb{C})$
 and $\gamma_1$ is a root for $\frak{sp}(1,\bb{C})$.
Write $\mu_i$ and $\nu_1$ for the fundamental weights corresponding to
 $\beta_i$ and $\gamma_1$, respectively.
Then the restriction $L^{\frak{g}}(3\omega_6)|_{\frak{g}^\sigma}$
 decomposes as 
\begin{align}
L^{\frak{g}}(3\omega_6)|_{\frak{g}^\sigma}
\simeq\bigoplus_{k=0}^{\infty}
L^{\frak{su}(5,1)}(k\mu_3+3\mu_5)\boxtimes
L^{\frak{sp}(1,\bb{R})}((k+3)\nu_1).
\end{align}

\smallskip
\subsection{$(\frak{g},\frak{g}^\sigma)=(\frak{e}_{7(-25)},\frak{e}_{6(-14)}\oplus\frak{so}(2))$}

Let $\alpha_i$ be as in Setting~\ref{e7(-25)}.
We may assume $\sigma=1$ on $\frak{t}$,
 $\sigma=1$ on $\frak{g}_{\alpha_i}$ for $i\neq 1,7$,
 and $\sigma=-1$ on $\frak{g}_{\alpha_i}$ for $i=1,7$.
We put
\begin{align*}  
&\beta_1:=\alpha_6,\quad \beta_2:=\alpha_3,\quad
\beta_3:=\alpha_5,\quad \beta_4:=\alpha_4,\quad \beta_5:=\alpha_2,\\
&\beta_6:=\alpha_1+\alpha_3+\alpha_4+\alpha_5+\alpha_6+\alpha_7.
\end{align*}
Then $\beta_i$ form a set of simple roots for $\frak{e}_{6, \bb{C}}$.
Write $\mu_i$ for the corresponding fundamental weights.
Let $e\in\frak{t}$ be the vector in the $\frak{so}(2)$-component
 of $\frak{g}^\sigma$ such that $\alpha_1(e)=\sqrt{-1}$.
Then the restriction $L^{\frak{g}}(4\omega_7)|_{\frak{g}^\sigma}$
 decomposes as 
\begin{align}
L^{\frak{g}}(4\omega_7)|_{\frak{g}^\sigma}
&\simeq\bigoplus_{k=0}^{\infty}
\Bigl(L^{\frak{e}_{6(-14)}}((k+4)\mu_6)\boxtimes \bb{C}_{k-2}\Bigr)
\oplus\bigoplus_{k=1}^{\infty}
\Bigl(L^{\frak{e}_{6(-14)}}(k\mu_1+4\mu_6)\boxtimes \bb{C}_{-k-2}\Bigr),
\end{align}
 where $\bb{C}_a$ is the character of $\frak{so}(2)$-component
 of $\frak{g}^\sigma$ on which $e$ acts as $\sqrt{-1}a$.

\smallskip
\subsection{$(\frak{g},\frak{g}^\sigma)=(\frak{e}_{7(-25)},
\frak{so}(2,10)\oplus\frak{sp}(1,\bb{R}))$}

Let $\alpha_i$ be as in Setting~\ref{e7(-25)}.
We may assume $\sigma=1$ on $\frak{t}$,
 $\sigma=1$ on $\frak{g}_{\alpha_i}$ for $i\neq 1$,
 and $\sigma=-1$ on $\frak{g}_{\alpha_1}$.
We put
\begin{align*}
\beta_i:=\alpha_{8-i}\ (1\leq i\leq 6),\quad
\gamma_1:=2\alpha_1+2\alpha_2+3\alpha_3+4\alpha_4+3\alpha_5+2\alpha_6+\alpha_7.
\end{align*}
Then $\beta_i$ form a set of simple roots for $\frak{so}(12,\bb{C})$
 and $\gamma_1$ is a root for $\frak{sp}(1,\bb{C})$.
Write $\mu_i$ and $\nu_1$ for the fundamental weights corresponding to
 $\beta_i$ and $\gamma_1$, respectively.
Then the restriction $L^{\frak{g}}(4\omega_7)|_{\frak{g}^\sigma}$
 decomposes as 
\begin{align}
L^{\frak{g}}(4\omega_7)|_{\frak{g}^\sigma}
&\simeq\bigoplus_{k=0}^{\infty}
L^{\frak{so}(2,10)}(4\mu_1+k\mu_5)\boxtimes
L^{\frak{sp}(1,\bb{R})}((k+4)\nu_1).
\end{align}

\smallskip
\subsection{$(\frak{g},\frak{g}^\sigma)=(\frak{e}_{7(-25)},\frak{su}(6,2))$}

Let $\alpha_i$ be as in Setting~\ref{e7(-25)}.
We may assume $\sigma=1$ on $\frak{t}$,
 $\sigma=1$ on $\frak{g}_{\alpha_i}$ for $i\neq 2$,
 and $\sigma=-1$ on $\frak{g}_{\alpha_2}$.
We put
\begin{align*}  
&\beta_1:=\alpha_1,\quad
 \beta_i:=\alpha_{i+1}\ (2\leq i\leq 6),\\
&\beta_7:=\alpha_1+2\alpha_2+2\alpha_3+3\alpha_4+2\alpha_5+\alpha_6.
\end{align*}
Then $\beta_i$ form a set of simple roots for $\frak{sl}(8,\bb{C})$.
Write $\mu_i$ for the corresponding fundamental weights.
Then the restriction $L^{\frak{g}}(4\omega_7)|_{\frak{g}^\sigma}$
 decomposes as 
\begin{align}
L^{\frak{g}}(4\omega_7)|_{\frak{g}^\sigma}
&\simeq\bigoplus_{k=0}^{\infty}
L^{\frak{su}(6,2)}(k\mu_4+4\mu_6).
\end{align}

\smallskip
\subsection{$(\frak{g},\frak{g}^\sigma)=(\frak{e}_{7(-25)},\frak{so}^*(12)\oplus\frak{su}(2))$}

Let $\alpha_i$ be as in Setting~\ref{e7(-25)}.
We may assume $\sigma=1$ on $\frak{t}$,
 $\sigma=1$ on $\frak{g}_{\alpha_i}$ for $i\neq 2,7$,
 and $\sigma=-1$ on $\frak{g}_{\alpha_i}$ for $i=2,7$.
We put
\begin{align*}
&\beta_i:=\alpha_{7-i}\ (1\leq i\leq 4),\quad
\beta_5:=\alpha_1,\quad 
\beta_6:=\alpha_2+\alpha_4+\alpha_5+\alpha_6+\alpha_7,\\
&\gamma_1:=\alpha_1+2\alpha_2+2\alpha_3+3\alpha_4+2\alpha_5+\alpha_6.
\end{align*}
Then $\beta_i$ form a set of simple roots for $\frak{so}(12,\bb{C})$
 and $\gamma_1$ is a root for $\frak{sl}(2,\bb{C})$.
Write $\mu_i$ and $\nu_1$ for the fundamental weights corresponding to
 $\beta_i$ and $\gamma_1$, respectively.
Then the restriction $L^{\frak{g}}(4\omega_7)|_{\frak{g}^\sigma}$
 decomposes as 
\begin{align}
\label{eq:e7sostar}
L^{\frak{g}}(4\omega_7)|_{\frak{g}^\sigma}
&\simeq\bigoplus_{k=0}^{\infty}
L^{\frak{so}^*(12)}(k\mu_5+4\mu_6)\boxtimes
F^{\frak{su}(2)}(k\nu_1).
\end{align}

\section{Zuckerman's derived functor modules}\label{sec:Zuc}
Most unitarizable highest weight modules are isomorphic to
 $A_{\frak{q}}(\lambda)$, also called Zuckerman's derived functor modules.
Let us fix some notation concerning Zuckerman's derived functor modules.  
Let $G$ be a connected reductive Lie group with a Cartan involution $\theta$.
We extend $\theta$ to a $\bb{C}$-linear involution on the complexified Lie algebra $\frak{g}_\bb{C}$ 
 and suppose that $\frak{q}$ is a $\theta$-stable parabolic subalgebra of 
 $\frak{g}_\bb{C}$.
The normalizer $L=N_G(\frak{q})$ of $\frak{q}$ is a connected reductive 
subgroup of $G$.
Hence a unitary character $\bb{C}_\lambda$ of $L$ is determined by 
its differential $\lambda\in\sqrt{-1}\frak{l}^*$. 
Associated to the data $(\frak{q},\lambda)$, one defines 
Zuckerman's derived functor module
 $A_\frak{q}(\lambda)$ as in \cite[(5.6)]{KnVo}.
In our normalization, $A_\frak{q}(0)$ is a unitarizable $(\frak {g}, K)$-module
 with non-zero $(\frak{g},K)$-cohomology,
 and in particular,
 has the same infinitesimal character 
as the trivial one-dimensional representation $\bb{C}$ of $\frak{g}$.

Let $\frak{u}$ be the nilradical of $\frak{q}$, so $\frak{q}=\frak{l}_\bb{C}+\frak{u}$.
Let $\frak{t}$ be a Cartan subalgebra of $\frak{k}$
 and $\frak{h}$ the centralizer of $\frak{t}$ in $\frak{g}$,
 which is a fundamental Cartan subalgebra of $\frak{g}$.
Choose a positive root system $\Delta^+(\frak{l}_\bb{C},\frak{h}_\bb{C})$
 for $\frak{l}$ and put
$\Delta^+(\frak{g}_\bb{C},\frak{h}_\bb{C})
:=\Delta^+(\frak{l}_\bb{C},\frak{h}_\bb{C})
 \cup \Delta(\frak{u},\frak{h}_\bb{C}).$
Denote by $\rho, \rho_{\frak{l}}$, and $\rho(\frak{u})\in\frak{h}_\bb{C}^*$
 half the sum of roots in
 $\Delta^+(\frak{g}_\bb{C},\frak{h}_\bb{C})$,
 $\Delta^+(\frak{l}_\bb{C},\frak{h}_\bb{C})$,
 and $\Delta(\frak{u},\frak{h}_\bb{C})$, respectively.
Let $\langle\cdot,\cdot\rangle$ be an invariant bilinear form 
on $\frak{h}_\bb{C}^*$ 
that is positive definite on the real span of the roots.
Following \cite[Definitions 0.49 and 0.52]{KnVo},
 for a unitary character $\bb{C}_\lambda$ of $L$, 
 we say $\lambda$ is in the {\it good range} if 
\[ \langle \lambda+\rho,\alpha\rangle>0,\qquad 
\alpha\in\Delta(\frak{u},\frak{h}_\bb{C}), \]
and in the {\it weakly fair range} if 
\[ \langle \lambda+\rho(\frak{u}),\alpha\rangle\geq 0,\qquad 
\alpha\in\Delta(\frak{u},\frak{h}_\bb{C}).\]
We state some basic properties of the $({\mathfrak{g}}, K)$-modules 
 $A_{\frak {q}}(\lambda)$.
\begin{thm}{\cite[Theorem 8.2 and Corollary 9.70]{KnVo}}\label{thm:aqfact}
If $\lambda$ is in the weakly fair range, $A_\frak{q}(\lambda)$
is unitarizable or zero.
If $\lambda$ is in the good range, $A_\frak{q}(\lambda)$ 
is non-zero and irreducible.
\end{thm}

We identify infinitesimal characters for $\frak{g}_\bb{C}$ with Weyl group orbits in $\frak{h}_\bb{C}^*$ and their representatives via the Harish-Chandra homomorphism.

\begin{thm}[{\cite[\S V.2]{KnVo}}]\label{thm:infchar}
$A_\frak{q}(\lambda)$ has the infinitesimal character $\lambda+\rho$.
\end{thm}

The $K$-type decomposition of $A_\frak{q}(\lambda)$,
 namely, the branching of the restriction of $A_\frak{q}(\lambda)$
 to $K$ is known as the generalized Blattner formula.
Choose a positive root system $\Delta^+(\frak{l}_\bb{C}\cap\frak{k}_\bb{C},\frak{t}_\bb{C})$ and set
$\Delta^+(\frak{k}_\bb{C},\frak{t}_\bb{C})
:=\Delta^+(\frak{l}_\bb{C}\cap\frak{k}_\bb{C},\frak{t}_\bb{C})
 \cup \Delta(\frak{u}\cap\frak{k}_\bb{C},\frak{t}_\bb{C})$
as a positive system for $\frak{k}$.
For a dominant integral weight $\mu\in\frak{t}_\bb{C}^*$ for $L\cap K$, 
 let $m(\mu)$ be the multiplicity of
 $F^{\frak{l}\cap\frak{k}}(\mu)$ in
 $S(\frak{u}\cap \frak{p}_\bb{C})
 \otimes  \bb{C}_{\lambda+2\rho(\frak{u}\cap\frak{p}_\bb{C})}$.
Write $\rho_K$ for half the sum of positive roots in $\frak{k}_\bb{C}$.
Then we have the following $K$-type formula:

\begin{thm}[{\cite[Equation (5.108a)]{KnVo}}]
\label{Blattner}
For weakly fair $\lambda$
\begin{align*}
\dim {\Hom}_{K}\bigl(A_\frak{q}(\lambda), F^\frak{k}(\mu) \bigr)
=\sum_{w} (-1)^{l(w)} m(w(\mu+\rho_K)-\rho_K),
\end{align*}
where the sum runs over the elements $w$ of the Weyl group of $K$
 such that $w(\mu+\rho_K)-\rho_K$
 is dominant for the positive system $\Delta^+(\frak{l}_\bb{C}\cap \frak{k}_\bb{C},\frak{t}_\bb{C})$
 and $l(w)$ denotes the length of $w$.
\end{thm}

\begin{rem}\label{rem:ktype}
Suppose that any weight $\mu$ with $m(\mu)\neq 0$
 is dominant for $\Delta^+(\frak{k}_\bb{C},\frak{t}_\bb{C})$.
Then it follows from the proof of \cite[Theorem 5.64]{KnVo} that 
\begin{align*}
\dim {\Hom}_{K}\bigl(A_\frak{q}(\lambda), F^\frak{k}(\mu) \bigr)
= m(\mu)
\end{align*}
even if $\lambda$ fails to be weakly fair.
\end{rem}

For particular $\frak{q}$, Zuckerman's modules
 $A_\frak{q}(\lambda)$ become highest weight modules.

\begin{de}
Suppose that $\frak{g}$ is a reductive Lie algebra of Hermitian type.
We say a $\theta$-stable parabolic subalgebra $\frak{q}$ of $\frak{g}_\bb{C}$
 is {\it holomorphic} if $\frak{q}\supset \frak{p}_-$. 
\end{de}

For holomorphic $\frak{q}$, 
 it follows that $\frak{u}\cap\frak{p}_\bb{C}\subset \frak{p}_-$ and
 hence all the eigenvalues of $\ad(z')$ in
 $\frak{u}\cap\frak{p}_\bb{C}$ are negative.
Therefore, Theorem \ref{Blattner} implies that each irreducible
 constituent of $A_\frak{q}(\lambda)$
 is a highest weight module.

The following proposition will be useful to see the irreducibility
 of $A_\frak{q}(\lambda)$.

\begin{prop}[{\cite[Theorem 8.31 and Proposition 8.75]{KnVo}}]
\label{abelnil}
For a $\theta$-stable parabolic subalgebra $\frak{q}$
 suppose that the nilpotent radical $\frak{u}$ is abelian.
Then $A_\frak{q}(\lambda)$ is irreducible or zero
 for weakly fair $\lambda$.
\end{prop}

Let $\frak{g}$ be a simple Lie algebra of Hermitian type 
 and choose simple roots as in Section~\ref{sec:pre}.
We write $\frak{q}(i)$ for the maximal parabolic subalgebra
 of $\frak{g}_\bb{C}$ corresponding to $\alpha_i$.
The minimal holomorphic representation for simple $\frak{g}$ is
 isomorphic to $A_\frak{q}(\lambda)$ with $\lambda$ in the weakly
 fair range if $\frak{g}=\frak{su}(m,n)$,
 $\frak{so}(2,2n)$ or $\frak{so}^*(2n)$.
In fact, we have
\begin{align*}
L^{\frak{su}(m,n)}(\omega_m) &\simeq A_{\frak{q}(1)}(-n\omega_1)
 \simeq A_{\frak{q}(m+n-1)}(-m\omega_{m+n-1}),\\
L^{\frak{so}(2,2n)}((n-1)\omega_1) &\simeq A_{\frak{q}(n)}(-2\omega_n)
 \simeq A_{\frak{q}(n+1)}(-2\omega_{n+1}),\\
L^{\frak{so}^*(2n)}(2\omega_n) &\simeq A_{\frak{q}(1)}(-(n-2)\omega_1).
\end{align*}

On the other hand, if $\frak{g}=\frak{so}(2,2n+1)$ for $(n\geq 1)$, $\frak{sp}(n,\bb{R})$ for $(n\geq 2)$, $\frak{e}_{6(-14)}$ or $\frak{e}_{7(-25)}$, the minimal holomorphic representation cannot be isomorphic to any $A_\frak{q}(\lambda)$, because the modules $A_\frak{q}(\lambda)$ do not attain the minimal Gelfand--Kirillov dimension.

\section{Proofs for holomorphic symmetric pairs}\label{sec:proofhol}

We now present three methods to prove the formulas in Section~\ref{sec:branchhol}. They are explained in the following three subsections \ref{subsec:dualpair}--\ref{subsec:Fock}.
Table~\ref{table:proof} shows which case can be treated with which of the three methods. We remark that, to avoid lengthy calculations, we do not prove each single branching law, but only demonstrate the three methods in some examples.

\begin{table}[!h]
\begin{center}
\begin{tabular}{ccccc}
\hline
$\qquad\frak{g}\qquad$&$\qquad\frak{g}^\sigma\qquad$
&dual pair&$A_\frak{q}(\lambda)$&Fock model\\
&
&\ref{subsec:dualpair}&\ref{subsec:aqlambda}&\ref{subsec:Fock}\\
\hline

$\frak{su}(m,n)$&$\frak{su}(p,q)\oplus\frak{su}(m-p,n-q)\oplus\frak{u}(1)$&
 $\bigcirc$&$\bigcirc$&\\
\hline

$\frak{su}(n,n)$&$\frak{so}^*(2n)$&
 $\bigcirc$&$\bigcirc$&\\
\hline

$\frak{su}(n,n)$&$\frak{sp}(n,\bb{R})$&
 $\bigcirc$&&\\
\hline

$\frak{so}(2,2n)$&$\frak{u}(1,n)$&
 &$\bigcirc$&\\
\hline

$\frak{so}(2,2n)$&$\frak{so}(2,m)\oplus\frak{so}(2n-m)$&
 &&$\bigcirc$\\
\hline

$\frak{so}(2,2n+1)$&$\frak{so}(2,m)\oplus\frak{so}(2n-m+1)$&
 &&$\bigcirc$\\
\hline

$\frak{so}^*(2n)$&$\frak{su}(m,n-m)\oplus\frak{u}(1)$&
 $\bigcirc$&$\bigcirc$&\\
\hline

$\frak{so}^*(2n)$&$\frak{so}^*(2m)\oplus\frak{so}^*(2n-2m)$&
 $\bigcirc$&$\bigcirc$&\\
\hline

$\frak{sp}(n,\bb{R})$&$\frak{su}(m,n-m)\oplus\frak{u}(1)$&
 $\bigcirc$&$\bigcirc$&\\
\hline

$\frak{sp}(n,\bb{R})$&$\frak{sp}(m,\bb{R})\oplus\frak{sp}(n-m,\bb{R})$&
 $\bigcirc$&&\\
\hline

$\frak{e}_{6(-14)}$&$\frak{so}(2,8)\oplus\frak{so}(2)$&
 &$\bigcirc$&\\
\hline

$\frak{e}_{6(-14)}$&$\frak{su}(4,2)\oplus\frak{su}(2)$&
 &$\bigcirc$&\\
\hline

$\frak{e}_{6(-14)}$&$\frak{so}^*(10)\oplus\frak{so}(2)$&
 &$\bigcirc$&\\
\hline

$\frak{e}_{6(-14)}$&$\frak{su}(5,1)\oplus\frak{sp}(1,\bb{R})$&
 &$\bigcirc$&\\
\hline

$\frak{e}_{7(-25)}$&$\frak{e}_{6(-14)}\oplus\frak{so}(2)$&
 &$\bigcirc$&\\
\hline

$\frak{e}_{7(-25)}$&$\frak{so}(2,10)\oplus\frak{sp}(1,\bb{R})$&
 &$\bigcirc$&\\
\hline

$\frak{e}_{7(-25)}$&$\frak{su}(6,2)$&
 &$\bigcirc$&\\
\hline

$\frak{e}_{7(-25)}$&$\frak{so}^*(12)\oplus\frak{su}(2)$&
 &$\bigcirc$&\\
\hline

\end{tabular}
\end{center}
\caption{Methods of proof for the branching laws for holomorphic symmetric pairs}
\label{table:proof}
\end{table}

\subsection{Seesaw dual pairs}\label{subsec:dualpair}
Some of the minimal holomorphic representations are isomorphic
 to the theta lift of a one-dimensional representation of a compact group.
In such cases, we can obtain branching laws by using the dual pair correspondence.
For details we refer the reader to \cite{Ad,How,Kud}.

Let $(G_1,H_1)$ be a dual pair of reductive groups in $Sp(N,\bb{R})$
 and suppose that $H_1$ is compact. Up to taking direct products, the
 possible dual pairs of this nature are
\begin{itemize}
\item $Sp(n,\bb{R}), O(m)\subset Sp(mn,\bb{R})$,
\item $U(m,n), U(p)\subset Sp((m+n)p,\bb{R})$,
\item $O^*(2n), Sp(m)\subset Sp(2mn,\bb{R})$.
\end{itemize}
 
The double covering groups $\widetilde{G_1}$ and $\widetilde{H_1}$ of 
 $G_1$ and $H_1$ in the metaplectic group $\widetilde{Sp(N,\bb{R})}$
 commute with each other.
We choose a Cartan involution of $\widetilde{Sp(N,\bb{R})}$ which induces
 a Cartan involution of $\widetilde{G_1}$ and 
 write $\widetilde{K_1}$ for the corresponding maximal compact subgroup.
Let $\omega$ be the Harish-Chandra module of
 the oscillator representation of the metaplectic group
 $\widetilde{Sp(N,\bb{R})}$.
Let ${\cal R}(\widetilde{H_1}, \omega)$ be the set of
 equivalence classes of irreducible representations of $\widetilde{H_1}$
 which occur in the restriction $\omega|_{\widetilde{H_1}}$.
Then $\omega$ is discretely decomposable as a
 $(\frak{g}_1\oplus\frak{h}_1,\widetilde{K_1}\times \widetilde{H_1})$-module:
\begin{align*}
\omega|_{(\frak{g}_1\oplus\frak{h}_1,\widetilde{K_1}\times \widetilde{H_1})}
 \simeq  \bigoplus_{\pi\in{\cal R}(\widetilde{H_1},\omega)}
 \theta(\pi)\boxtimes \pi.
\end{align*}
The $(\frak{g}_1,\widetilde{K_1})$-module $\theta(\pi)$ is irreducible by \cite{How} and called the local theta lift of $\pi$.
For the three irreducible reductive dual pairs above explicit correspondences
 $\theta(\pi)\leftrightarrow \pi$ are given in \cite{Ad}.
Write $y\in \widetilde{Sp(N,\bb{R})}$ for the element
 that is not equal to the identity element and mapped to the identity
 by the covering map $\widetilde{Sp(N,\bb{R})}\to Sp(N,\bb{R})$.
We say a representation $\pi$ of $\widetilde{H_1}$ is genuine if $\pi(y)=-1$.
Since $\omega(y)=-1$, $\pi\in{\cal R}(\widetilde{H_1},\omega)$
 implies that $\pi$ is genuine.

Let $(G_2,H_2)$ be another reductive dual pair in $Sp(N,\bb{R})$
 such that $G_1\supset G_2$, $H_1\subset H_2$, and $H_2$ is compact.
Such reductive dual pairs $(G_1,H_1)$ and $(G_2,H_2)$ are called seesaw dual pairs (\cite{Kud}):
\[
\xymatrix{
G_1 \ar@{-}[d] \ar@{-}[dr] & H_2 \ar@{-}[d] \ar@{-}[dl] \\
G_2  & H_1 
}
\]
Then the restriction 
 $\omega|_{(\frak{g}_2\oplus\frak{h}_1,\widetilde{K_2}\times \widetilde{H_1})}$
 can be written in two different ways:
\begin{align*}
\omega|_{(\frak{g}_2\oplus\frak{h}_1,\widetilde{K_2}\times \widetilde{H_1})}
 \simeq  \bigoplus_{\pi\in{\cal R}(\widetilde{H_1},\omega)}
 \theta(\pi)|_{(\frak{g}_2,\widetilde{K_2})}\boxtimes \pi
 \simeq  \bigoplus_{\rho\in{\cal R}(\widetilde{H_2},\omega)}
 \theta(\rho)\boxtimes \rho|_{\widetilde{H_1}}.
\end{align*}
We therefore get
\begin{align*}
\theta(\pi)|_{(\frak{g}_2,\widetilde{K_2})}
\simeq  \bigoplus_{\rho\in{\cal R}(\widetilde{H_2},\omega)}
 \theta(\rho)^{\oplus m(\pi,\rho)},
\end{align*}
where $m(\pi,\rho):=
 \dim\Hom_{\widetilde{H_1}}(\pi,\rho|_{\widetilde{H_1}})$.
 
This observation can be used to find explicit branching laws for minimal holomorphic representations.
We illustrate this technique in two cases.

\subsubsection{$(\frak{g},\frak{g}^\sigma)=(\frak{u}(m,n),\frak{u}(p,q)\oplus\frak{u}(m-p,n-q))$}
Choose a standard basis $\epsilon_1,\dots,\epsilon_{m+n} \in \frak{t}_\bb{C}^*$
so that 
\begin{align*}
\Delta^+(\frak{k}_\bb{C},\frak{t}_\bb{C})
 &=\{\epsilon_i-\epsilon_j\}_{1\leq i < j  \leq m}
 \cup\{\epsilon_{m+i}-\epsilon_{m+j}\}_{1\leq i < j \leq n},\\
\Delta(\frak{p}_+,\frak{t}_\bb{C})
 &=\{\epsilon_i-\epsilon_{m+j}\}
 _{1\leq i\leq m,\;1\leq j \leq n}.
\end{align*}
We may assume that $\frak{t}\subset \frak{g}^\sigma$ and
\begin{align*}
\Delta^+(\frak{k}^\sigma_\bb{C},\frak{t}_\bb{C})
 ={}& \{\epsilon_i-\epsilon_j\}_{1\leq i < j  \leq p}
 \cup\{\epsilon_{p+i}-\epsilon_{p+j}\}_{1\leq i < j \leq m-p}\\
& \cup\{\epsilon_{m+i}-\epsilon_{m+j}\}_{1\leq i < j \leq n-q}
 \cup\{\epsilon_{m+n-q+i}-\epsilon_{m+n-q+j}\}_{1\leq i < j \leq q},\\
\Delta(\frak{p}^\sigma_+,\frak{t}_\bb{C})
 ={}& \{\epsilon_i-\epsilon_{m+n-q+j}\} _{1\leq i\leq p,\;1\leq j \leq q}
 \cup\{\epsilon_{p+i}-\epsilon_{m+j}\}_{1\leq i\leq m-p,\;1\leq j \leq n-q}.
\end{align*}
Put $G_1:=U(m,n)$, $H_1:=U(1)$, $G_2:=U(p,q)\times U(m-p,n-q)$,
 and $H_2:=U(1)\times U(1)$
 such that $(G_1,H_1)$ and $(G_2,H_2)$ form seesaw dual pairs in $Sp(m+n,\bb{R})$.
The covering group $\widetilde{H_1}$ splits if and only if $m+n$ is even.
In any case a representation $\pi$ of $\widetilde{H_1}$ is determined by 
 the Lie algebra action provided that $\pi$ is genuine.
Hence the genuine representations of $\widetilde{H_1}$ are given
 by $\det^{\frac{m+n}{2}+k}$ for $k\in \bb{Z}$.
We have 
\begin{align*}
\theta({\det}^{\frac{m-n}{2}})\simeq
 L^{\frak{u}(m,n)}\Bigl(\frac{1}{2}
 \bigl((\epsilon_1+\cdots+\epsilon_m)
 -(\epsilon_{m+1}+\cdots+\epsilon_{m+n})\bigr)\Bigr).
\end{align*}

Suppose first that $p,q,m-p,n-q\geq 1$.
Then a genuine representation ${\det}^a\boxtimes {\det}^b$ of
 $\widetilde{H_2}$ belongs to ${\cal R}(\widetilde{H_2},\omega)$
 if and only if $a-\frac{p+q}{2}\in \bb{Z}$ and
 $b-\frac{m+n-p-q}{2}\in \bb{Z}$.
Since $H_1$ is diagonally embedded in $H_2$, we have
\begin{align*}
\Hom_{\widetilde{H_1}}
 ({\det}^{\frac{m-n}{2}}, ({\det}^a\boxtimes {\det}^b)|_{\widetilde{H_1}})
 \simeq \begin{cases} \bb{C} & \text{if $a+b=\frac{m-n}{2}$,} \\
 0 & \text{otherwise.} \end{cases}
\end{align*}
For ${\det}^{\frac{p-q}{2}-k}\boxtimes {\det}^{\frac{(m-p)-(n-q)}{2}+k}\in
 {\cal R}(\widetilde{H_2},\omega)$, we have
\begin{align*}
&\theta({\det}^{\frac{p-q}{2}-k}\boxtimes {\det}^{\frac{(m-p)-(n-q)}{2}+k})\\
&\quad\simeq 
 L^{\frak{u}(p,q)}\Bigl(-k\epsilon_{m+n}+
 \frac{1}{2}\bigl((\epsilon_1+\cdots+\epsilon_p)
 -(\epsilon_{m+n-q+1}+\cdots+\epsilon_{m+n})\bigr)\Bigr)\\
&\qquad\boxtimes  L^{\frak{u}(m-p,n-q)}\Bigl(k\epsilon_{p+1}+
 \frac{1}{2}\bigl((\epsilon_{p+1}+\cdots+\epsilon_{m})
 -(\epsilon_{m+1}+\cdots+\epsilon_{m+n-q})\bigr)\Bigr)
\end{align*}
if $k\geq 0$ and 
\begin{align*}
&\theta({\det}^{\frac{p-q}{2}-k}\boxtimes {\det}^{\frac{(m-p)-(n-q)}{2}+k})\\
&\quad\simeq 
 L^{\frak{u}(p,q)}\Bigl(-k\epsilon_1+
 \frac{1}{2}\bigl((\epsilon_1+\cdots+\epsilon_p)
 -(\epsilon_{m+n-q+1}+\cdots+\epsilon_{m+n})\bigr)\Bigr)\\
&\qquad \boxtimes  L^{\frak{u}(m-p,n-q)}\Bigl(k\epsilon_{m+n-q}+
 \frac{1}{2}\bigl((\epsilon_{p+1}+\cdots+\epsilon_{m})
 -(\epsilon_{m+1}+\cdots+\epsilon_{m+n-q})\bigr)\Bigr)
\end{align*}
if $k<0$.
As a consequence,
\begin{align*}
&L^{\frak{u}(m,n)}\Bigl(\frac{1}{2}\bigl((\epsilon_1+\cdots+\epsilon_m)
 -(\epsilon_{m+1}+\cdots+\epsilon_{m+n})\bigr)\Bigr)
 \Bigr|_{\frak{u}(p,q)\oplus\frak{u}(m-p,n-q)}\\
&\simeq
\bigoplus_{k=0}^{\infty}
 \Bigg(L^{\frak{u}(p,q)}\Bigl(-k\epsilon_{m+n}+
 \frac{1}{2}\bigl((\epsilon_1+\cdots+\epsilon_p)
 -(\epsilon_{m+n-q+1}+\cdots+\epsilon_{m+n})\bigr)\Bigr)\\
&\qquad\boxtimes  L^{\frak{u}(m-p,n-q)}\Bigl(k\epsilon_{p+1}+
 \frac{1}{2}\bigl((\epsilon_{p+1}+\cdots+\epsilon_{m})
 -(\epsilon_{m+1}+\cdots+\epsilon_{m+n-q})\bigr)\Bigr)\Bigg)\\
&\quad\oplus\bigoplus_{k=1}^{\infty} 
\Bigg(L^{\frak{u}(p,q)}\Bigl(k\epsilon_1+
 \frac{1}{2}\bigl((\epsilon_1+\cdots+\epsilon_p)
 -(\epsilon_{m+n-q+1}+\cdots+\epsilon_{m+n})\bigr)\Bigr)\\
&\qquad\boxtimes  L^{\frak{u}(m-p,n-q)}\Bigl(-k\epsilon_{m+n-q}+
 \frac{1}{2}\bigl((\epsilon_{p+1}+\cdots+\epsilon_{m})
 -(\epsilon_{m+1}+\cdots+\epsilon_{m+n-q})\bigr)\Bigr)\Bigg).
\end{align*}

Suppose next that $p,q,m-p\geq 1$ and $n=q$.
Then a genuine representation ${\det}^a\boxtimes {\det}^b$ of
 $\widetilde{H_2}$ belongs to ${\cal R}(\widetilde{H_2},\omega)$
 if and only if $b-\frac{m-p}{2}\in \bb{Z}_{\geq 0}$.
Hence we obtain
\begin{align*}
&L^{\frak{u}(m,n)}\Bigl(\frac{1}{2}\bigl((\epsilon_1+\cdots+\epsilon_m)
 -(\epsilon_{m+1}+\cdots+\epsilon_{m+n})\bigr)\Bigr)
 \Bigr|_{\frak{u}(p,n)\oplus\frak{u}(m-p)}\\
&\simeq
\bigoplus_{k=0}^{\infty}
 L^{\frak{u}(p,n)}\Bigl(-k\epsilon_{m+n}+
 \frac{1}{2}\bigl((\epsilon_1+\cdots+\epsilon_p)
 -(\epsilon_{m+1}+\cdots+\epsilon_{m+n})\bigr)\Bigr)\\
&\qquad\boxtimes  F^{\frak{u}(m-p)}\Bigl(k\epsilon_{p+1}+
 \frac{1}{2}(\epsilon_{p+1}+\cdots+\epsilon_{m})\Bigr).
\end{align*}

These imply formulas \eqref{eq:umnpq} and \eqref{eq:umnpn}.

\subsubsection{}
Let $(\frak{g},\frak{g}^\sigma)=
 (\frak{so}^*(2n),\frak{u}(m,n-m))$.
Choose a standard basis $\epsilon_1,\dots,\epsilon_{n} \in \frak{t}_\bb{C}^*$
so that 
\begin{align*}
\Delta^+(\frak{k}_\bb{C},\frak{t}_\bb{C})
 &=\{\epsilon_i-\epsilon_j\}_{1\leq i < j  \leq n},\\
\Delta(\frak{p}_+,\frak{t}_\bb{C})
 &=\{\epsilon_i+\epsilon_{j}\}_{1\leq i< j \leq n}.
\end{align*}
We may assume that $\frak{t}\subset \frak{g}^\sigma$ and
\begin{align*}
\Delta^+(\frak{k}^\sigma_\bb{C},\frak{t}_\bb{C})
 &=\{\epsilon_i-\epsilon_j\}_{1\leq i < j  \leq m}
 \cup\{\epsilon_{m+i}-\epsilon_{m+j}\}_{1\leq i < j \leq n-m},\\
\Delta(\frak{p}^\sigma_+,\frak{t}_\bb{C})
 &=\{\epsilon_i+\epsilon_{m+j}\} _{1\leq i\leq m,\;1\leq j \leq n-m}.
\end{align*}
Put $G_1:=O^*(2n)$, $H_1:=Sp(1)$, $G_2:=U(m,n-m)$, and $H_2:=U(2)$
 such that $(G_1,H_1)$ and $(G_2,H_2)$ form seesaw dual pairs in $Sp(2n,\bb{R})$.
Let $\frak{c}$ be a Cartan subalgebra of $\frak{h}_2$
 and $\delta_1,\delta_2$ be the standard basis of $\frak{c}^*$
The genuine representations $\pi$ of $\widetilde{H_2}$ are determined by 
 the Lie algebra actions and given
 by $a\delta_1+b\delta_2$ such that $a-b\in \bb{Z}_{\geq 0}$ and $a\in \frac{n}{2}+\bb{Z}$.
Since $\widetilde{H_1}\simeq Sp(1)\times \bb{Z}/2\bb{Z}$,
 there exists a unique non-trivial character of $\widetilde{H_1}$,
 which we denote by $\chi$.
We have
\begin{align*}
\theta(\chi)\simeq
 L^{\frak{so}^*(2n)}(\epsilon_1+\cdots+\epsilon_n).
\end{align*}

Suppose first that $m,n-m\geq 2$.
Then any genuine representation $F^{\frak{u}(2)}(a\delta_1+b\delta_2)$ of
 $\widetilde{H_2}$ belongs to ${\cal R}(\widetilde{H_2},\omega)$.
We can see that any irreducible representation of $\widetilde{H_2}$
 remains irreducible when restricted to $\widetilde{H_1}$ and hence
\begin{align*}
\Hom_{\widetilde{H_1}}
 (\chi, F^{\frak{u}(2)}(a\delta_1+b\delta_2)|_{\widetilde{H_1}})
 \simeq \begin{cases} \bb{C} &
 \text{if $a=b$,} \\
 0 & \text{if $a>b$.} \end{cases}
\end{align*}
For $F^{\frak{u}(2)}((m-\frac{n}{2}+k)
 (\delta_1+\delta_2))\in{\cal R}(\widetilde{H_2},\omega)$, we have
\begin{multline*}
\theta\Bigl(F^{\frak{u}(2)}
 \Bigl(\Bigl(m-\frac{n}{2}+k\Bigr)(\delta_1+\delta_2)\Bigr)\Bigr)\\
\simeq 
 L^{\frak{u}(m,n-m)}(k(\epsilon_{n-1}+\epsilon_n)
 +(\epsilon_1+\cdots+\epsilon_m)-(\epsilon_{m+1}+\cdots+\epsilon_n))
\end{multline*}
if $k\leq 0$ and 
\begin{multline*}
\theta\Bigl(F^{\frak{u}(2)}
 \Bigl(\Bigl(m-\frac{n}{2}+k\Bigr)(\delta_1+\delta_2)\Bigr)\Bigr)\\
\simeq 
 L^{\frak{u}(m,n-m)}(k(\epsilon_1+\epsilon_2)
 +(\epsilon_1+\cdots+\epsilon_m)
 -(\epsilon_{m+1}+\cdots+\epsilon_n))
\end{multline*}
if $k>0$.
As a consequence,
\begin{align*}
& L^{\frak{so}^*(2n)}(\epsilon_1+\cdots+\epsilon_n)
 |_{\frak{u}(m,n-m)}\\
&\quad\simeq
\bigoplus_{k=0}^{\infty}
 L^{\frak{u}(m,n-m)}(-k(\epsilon_{n-1}+\epsilon_n)
 +(\epsilon_1+\cdots+\epsilon_m)
 -(\epsilon_{m+1}+\cdots+\epsilon_n))\\
&\qquad\oplus\bigoplus_{k=1}^{\infty} 
 L^{\frak{u}(m,n-m)}(k(\epsilon_1+\epsilon_2)
 +(\epsilon_1+\cdots+\epsilon_m)-(\epsilon_{m+1}+\cdots+\epsilon_n)).
\end{align*}

Suppose next that $m=1$ and $n-m\geq 2$.
Then a genuine representation $F^{\frak{u}(2)}(a\delta_1+b\delta_2)$ of
 $\widetilde{H_2}$ belongs to ${\cal R}(\widetilde{H_2},\omega)$
 if and only if $b\leq 1-\frac{n}{2}$.
Hence
\begin{multline*}
L^{\frak{so}^*(2n)}(\epsilon_1+\cdots+\epsilon_n)
 |_{\frak{u}(1,n-1)}\\
\simeq
\bigoplus_{k=0}^{\infty}
 L^{\frak{u}(1,n-1)}(-k(\epsilon_{n-1}+\epsilon_n)
 +\epsilon_1-(\epsilon_2+\cdots+\epsilon_n)).
\end{multline*}
Similarly for $n-m=1$.

These imply formulas \eqref{eq:sostaru}, \eqref{eq:sostaru1n-1}, and \eqref{eq:sostarun-11}.

\subsection{Zuckerman's derived functor modules $A_\frak{q}(\lambda)$}\label{subsec:aqlambda}
We will see that the restriction of the minimal holomorphic representation
 with respect to holomorphic symmetric pairs $(\frak{g},\frak{g}^\sigma)$
 can be written as a direct sum
 of $A_\frak{q}(\lambda)$ for a maximal parabolic subalgebra
 $\frak{q}\subset\frak{g}^\sigma_\bb{C}$
 unless $(\frak{g},\frak{g}^\sigma)
 =(\frak{sp}(n,\bb{R}),\frak{sp}(m,\bb{R})\oplus\frak{sp}(n-m,\bb{R}))$.
However, not in all cases we can tell from the general results Theorem~\ref{thm:aqfact} and Proposition~\ref{abelnil} whether the occurring modules $A_\frak{q}(\lambda)$ are irreducible, see Table~\ref{table:proof} for the list of cases where this works.
In fact, the occurring $A_\frak{q}(\lambda)$ is reducible in some cases, see Remark~\ref{rem:reducibleaq}.

The following formulas \eqref{eq:umnpqaq} -- \eqref{eq:e7sostaraq}
 correspond to \eqref{eq:umnpq} -- \eqref{eq:e7sostar}, respectively.
We follow the corresponding subsections in Section~\ref{sec:branchhol}
 for the notation $\omega_i$, $\mu_i$, $\nu_i$, and $\bb{C}_a$.
If a simple factor $\frak{g}'_\bb{C}$ of $\frak{g}^\sigma_\bb{C}$
 has simple roots $\beta_1,\beta_2,...$,
 we write $\frak{q}'(i)$ for the maximal parabolic subalgebra of
 $\frak{g}'_\bb{C}$ corresponding to $\beta_i$.
Similarly, if a factor $\frak{g}''_\bb{C}$ has simple roots $\gamma_1,\gamma_2,...$,
 write $\frak{q}''(i)$ for the maximal parabolic subalgebra
 corresponding to $\gamma_i$.

\subsubsection{$(\frak{g},\frak{g}^\sigma)
 =(\frak{su}(m,n),\frak{su}(p,q)\oplus\frak{su}(m-p,n-q)\oplus\frak{u}(1))$}
Assume $m\geq n$. Then we have
\begin{align}
\label{eq:umnpqaq}
&L^{\frak{g}}(\omega_m)|_{\frak{g}^\sigma}\\ \nonumber
&\simeq\bigoplus_{\substack{\frac{p-q}{2}\leq k\\ k\in\bb{Z}}}
A_{\frak{q}'(p+q-1)}((k-p)\mu_{p+q-1})\boxtimes
A_{\frak{q}''(1)}((k-(n-q))\nu_1)\boxtimes
\bb{C}_{-k+\frac{np-mq}{m+n}}
\\ \nonumber
&\quad\oplus\bigoplus_{\substack{\frac{(n-q)-(m-p)}{2}
 \leq k<\frac{p-q}{2}\\ k\in\bb{Z}}}
A_{\frak{q}'(1)}((-k-q)\mu_1)\boxtimes
A_{\frak{q}''(1)}((k-(n-q))\nu_1)\boxtimes
\bb{C}_{-k+\frac{np-mq}{m+n}}
\\ \nonumber
&\quad\oplus\!\!\!\!\!\!\!
\bigoplus_{\substack{k<\frac{(n-q)-(m-p)}{2}\\ k\in\bb{Z}}}
\!\!\!\!\!\!\!\!A_{\frak{q}'(1)}((-k-q)\mu_1)\boxtimes
A_{\frak{q}''(m+n-p-q-1)}((-k-(m-p))\nu_{m+n-p-q-1})
\boxtimes\bb{C}_{-k+\frac{np-mq}{m+n}}
\end{align}
if $p,q,m-p,n-q\geq 1$
and 
\begin{align}
\label{eq:umnpnaq}
&L^{\frak{g}}(\omega_m)|_{\frak{g}^\sigma}\\ \nonumber
\simeq
&\bigoplus_{\substack{0,\frac{p-q}{2}\leq k\\ k\in\bb{Z}}}
A_{\frak{q}'(p+q-1)}((k-p)\mu_{p+q-1})\boxtimes
F^{\frak{su}(m-p)}(k\nu_1)\boxtimes
\bb{C}_{-k+\frac{n(p-m)}{m+n}}\\ \nonumber
&\oplus\bigoplus_{\substack{0\leq k<\frac{p-q}{2}\\ k\in\bb{Z}}}
A_{\frak{q}'(1)}((-k-q)\mu_1)\boxtimes
F^{\frak{su}(m-p)}(k\nu_1)\boxtimes
\bb{C}_{-k+\frac{n(p-m)}{m+n}}
\end{align}
if $n=q$ and $p,q,m-p\geq 1$.

\subsubsection{$(\frak{g},\frak{g}^\sigma)
 =(\frak{su}(n,n),\frak{so}^*(2n))$}
We have
\begin{align}\label{eq:unnsostaraq}
L^{\frak{g}}(\omega_n)|_{\frak{g}^\sigma}
\simeq
\bigoplus_{k=0}^{\infty} A_{\frak{q}'(1)}((2k-(n-2))\mu_1).
\end{align}

\subsubsection{$(\frak{g},\frak{g}^\sigma)
 =(\frak{su}(n,n),\frak{sp}(n,\bb{R}))$}
We have
\begin{align}\label{eq:unnspnaq}
L^{\frak{g}}(\omega_n)|_{\frak{g}^\sigma}
\simeq A_{\frak{q}'(1)}(-n\mu_1).
\end{align}

\subsubsection{$(\frak{g},\frak{g}^\sigma)
 =(\frak{so}(2,2n),\frak{su}(1,n)\oplus\frak{u}(1))$}
We have
\begin{align}
\label{eq:souaq}
L^{\frak{g}}((n-1)\omega_1)|_{\frak{g}^\sigma}
\simeq
\bigoplus_{k=0}^{\infty} A_{\frak{q}'(2)}((k-1)\mu_2)
\boxtimes \bb{C}_{k+\frac{n-1}{2}}.
\end{align}

\subsubsection{$(\frak{g},\frak{g}^\sigma)
 =(\frak{so}(2,2n),\frak{so}(2,m)\oplus\frak{so}(2n-m))$}
For $m=2l$ even we have
\begin{align}
L^{\frak{g}}((n-1)\omega_1)|_{\frak{g}^\sigma}
\simeq
\bigoplus_{k=0}^{\infty} A_{\frak{q}'(1)}((n-2l+k-1)\mu_1)
\boxtimes F^{\frak{so}(2n-2l)}(k\nu_1).
\end{align}
if $n-l\geq 2$ and 
\begin{align}
L^{\frak{g}}((n-1)\omega_1)|_{\frak{g}^\sigma}
\simeq
\bigoplus_{k=-\infty}^{\infty} A_{\frak{q}'(1)}((-n+|k|+1)\mu_1)
\boxtimes \bb{C}_k
\end{align}
if $n-l=1$.

For $m=2l+1$ odd we have
\begin{align}
L^{\frak{g}}((n-1)\omega_1)|_{\frak{g}^\sigma}
\simeq
\bigoplus_{k=0}^{\infty} A_{\frak{q}'(1)}((n-2l+k-2)\mu_1)
\boxtimes F^{\frak{so}(2n-2l-1)}(k\nu_1)
\end{align}
if $n-l\geq 2$ and
\begin{align}\label{eq:so2eso2e-1aq}
L^{\frak{g}}((n-1)\omega_1)|_{\frak{g}^\sigma}
\simeq
A_{\frak{q}'(l+1)}(-2\mu_{l+1})
\end{align}
if $n-l=1$.

\subsubsection{$(\frak{g},\frak{g}^\sigma)
 =(\frak{so}(2,2n+1),\frak{so}(2,m)\oplus\frak{so}(2n-m+1))$}
 For $m=2l$ even we have
\begin{align}
L^{\frak{g}}\Bigl(\Bigl(n-\frac{1}{2}\Bigr)\omega_1\Bigr)
\Bigl|_{\frak{g}^\sigma}
\simeq
\bigoplus_{k=0}^{\infty} A_{\frak{q}'(1)}
 \Bigl(\Bigl(n-2l+k-\frac{1}{2}\Bigr)\mu_1\Bigr)
\boxtimes F^{\frak{so}(2n-2l+1)}(k\nu_1)
\end{align}
if $n>l$ and 
\begin{align}
L^{\frak{g}}\Bigl(\Bigl(n-\frac{1}{2}\Bigr)\omega_1\Bigr)
\Bigl|_{\frak{g}^\sigma}
\simeq A_{\frak{q}'(1)}\Bigl(\Bigl(-n-\frac{1}{2}\Bigr)\mu_1\Bigr)\oplus
 A_{\frak{q}'(1)}\Bigl(\Bigl(-n+\frac{1}{2}\Bigr)\mu_1\Bigr)
\end{align}
if $n=l$.

For $m=2l+1$ odd we have
\begin{align}
L^{\frak{g}}\Bigl(\Bigl(n-\frac{1}{2}\Bigr)\omega_1\Bigr)
\Bigl|_{\frak{g}^\sigma}
\simeq
\bigoplus_{k=0}^{\infty} A_{\frak{q}'(1)}
\Bigl(\Bigl(n-2l+k-\frac{3}{2}\Bigr)\mu_1\Bigr)
\boxtimes F^{\frak{so}(2n-2l)}(k\nu_1)
\end{align}
if $n-l\geq 2$ and 
\begin{align}
L^{\frak{g}}\Bigl(\Bigl(n-\frac{1}{2}\Bigr)\omega_1\Bigr)
\Bigl|_{\frak{g}^\sigma}
\simeq
\bigoplus_{k=-\infty}^{\infty} A_{\frak{q}'(1)}
\Bigl(\Bigl(-n+|k|+\frac{1}{2}\Bigr)\mu_1\Bigr)
\boxtimes \bb{C}_k
\end{align}
if $n-l=1$.

\subsubsection{$(\frak{g},\frak{g}^\sigma)
 =(\frak{so}^*(2n),\frak{su}(m,n-m)\oplus\frak{u}(1))$}
We have
\begin{align}\label{eq:sostaruaq}
L^{\frak{g}}(2\omega_n)|_{\frak{g}^\sigma}
\simeq{}& \bigoplus_{\substack{m-\frac{n}{2}\leq k\\ k\in \bb{Z}}}
A_{\frak{q}'(n-2)}((k-m)\mu_{n-2})\boxtimes \bb{C}_{-k-\frac{n}{2}+m}\\
 \nonumber
&\oplus\bigoplus_{\substack{k<m-\frac{n}{2}\\ k\in \bb{Z}}}
A_{\frak{q}'(2)}((-k-(n-m))\mu_2)\boxtimes \bb{C}_{-k-\frac{n}{2}+m}
\end{align}
if $2\leq m\leq n-2$, 
\begin{align}\label{eq:sostaru1n-1aq}
L^{\frak{g}}(2\omega_n)|_{\frak{g}^\sigma}
\simeq\bigoplus_{k=0}^{\infty}
A_{\frak{q}'(n-2)}((k-1)\mu_{n-2})\boxtimes \bb{C}_{-k-\frac{n}{2}+1}
\end{align}
if $m=1$, and
\begin{align}\label{eq:sostarun-11aq}
L^{\frak{g}}(2\omega_n)|_{\frak{g}^\sigma}
\simeq\bigoplus_{k=0}^{\infty}
A_{\frak{q}'(2)}((k-1)\mu_2)\boxtimes \bb{C}_{k+\frac{n}{2}-1}
\end{align}
if $m=n-1$.

\subsubsection{$(\frak{g},\frak{g}^\sigma)
 =(\frak{so}^*(2n),\frak{so}^*(2m)\oplus\frak{so}^*(2n-2m))$}
 We have
\begin{align}\label{eq:sostarsostaraq}
L^{\frak{g}}(2\omega_n)|_{\frak{g}^\sigma}
\simeq
\bigoplus_{k=0}^{\infty} A_{\frak{q}'(1)}((k-(m-2))\mu_1)
\boxtimes  A_{\frak{q}''(1)}((k-(n-m-2))\nu_1)
\end{align}
if $2\leq m\leq n-2$, 
\begin{align}\label{eq:sostarsostar1n-1aq}
L^{\frak{g}}(2\omega_n)|_{\frak{g}^\sigma}
\simeq
\bigoplus_{k=0}^{\infty} \bb{C}_{k+1}
\boxtimes  A_{\frak{q}''(1)}((k-(n-3))\nu_1)
\end{align}
if $m=1$, and
\begin{align}\label{eq:sostarsostarn-11aq}
L^{\frak{g}}(2\omega_n)|_{\frak{g}^\sigma}
\simeq
\bigoplus_{k=0}^{\infty} 
A_{\frak{q}'(1)}((k-(n-3))\mu_1) \boxtimes \bb{C}_{k+1}
\end{align}
if $n-m=1$.

\subsubsection{$(\frak{g},\frak{g}^\sigma)
 =(\frak{sp}(n,\bb{R}),\frak{su}(m,n-m)\oplus\frak{u}(1))$}
 We have
\begin{align}\label{eq:spuaq}
L^{\frak{g}}\Bigl(\frac{1}{2}\omega_n\Bigr)\Bigl|_{\frak{g}^\sigma}
\simeq{}&\bigoplus_{\substack{m-\frac{n}{2}\leq 2k\\ k\in \bb{Z}}}
A_{\frak{q}'(n-1)}((2k-m)\mu_{n-1})
 \boxtimes \bb{C}_{-k-\frac{n}{4}+\frac{m}{2}}\\ \nonumber
&\oplus\bigoplus_{\substack{2k<m-\frac{n}{2}\\ k\in \bb{Z}}}
A_{\frak{q}'(1)}((-2k-(n-m))\mu_1)
 \boxtimes \bb{C}_{-k-\frac{n}{4}+\frac{m}{2}}.
\end{align}

\subsubsection{$(\frak{g},\frak{g}^\sigma)=(\frak{sp}(n,\bb{R}),\frak{sp}(m,\bb{R})\oplus\frak{sp}(n-m,\bb{R}))$}

In this case it is not possible to write the restriction $L^{\frak{g}}(\frac{1}{2}\omega_n)|_{\frak{g}^\sigma}$ as a direct sum of Zuckerman's derived functor modules.
\setcounter{equation}{21}

\subsubsection{$(\frak{g},\frak{g}^\sigma)
 =(\frak{e}_{6(-14)},\frak{so}(2,8)\oplus\frak{so}(2))$}
 We have
\begin{align}
\label{eq:e6soaq}
L^{\frak{g}}(3\omega_6)|_{\frak{g}^\sigma}
&\simeq\bigoplus_{k=0}^{\infty}
A_{\frak{q}'(5)}((k-2)\mu_5)\boxtimes \bb{C}_{k+2}.
\end{align}

\subsubsection{$(\frak{g},\frak{g}^\sigma)
 =(\frak{e}_{6(-14)},\frak{su}(4,2)\oplus\frak{su}(2))$}
 We have
\begin{align}
L^{\frak{g}}(3\omega_6)|_{\frak{g}^\sigma}
&\simeq\bigoplus_{k=0}^{\infty}
A_{\frak{q}'(3)}((k-2)\mu_3)\boxtimes
F^{\frak{su}(2)}(k\nu_1).
\end{align}

\subsubsection{$(\frak{g},\frak{g}^\sigma)
 =(\frak{e}_{6(-14)},\frak{so}^*(10)\oplus\frak{so}(2))$}
 We have
\begin{align}
L^{\frak{g}}(3\omega_6)|_{\frak{g}^\sigma}
\simeq\bigoplus_{k=1}^{\infty}
\Bigl(A_{\frak{q}'(5)}((k-5)\mu_5)\boxtimes \bb{C}_{k-1}\Bigr)
\oplus\bigoplus_{k=0}^{\infty}
\Bigl(A_{\frak{q}'(4)}((k-3)\mu_4)\boxtimes \bb{C}_{-k-1}\Bigr).
\end{align}

\subsubsection{$(\frak{g},\frak{g}^\sigma)
 =(\frak{e}_{6(-14)},\frak{su}(5,1)\oplus\frak{sp}(1,\bb{R}))$}
 We have
\begin{align}
L^{\frak{g}}(3\omega_6)|_{\frak{g}^\sigma}
&\simeq\bigoplus_{k=0}^{\infty}
A_{\frak{q}'(3)}((k-1)\mu_3)\boxtimes A_{\frak{q}''(1)}((k+1)\nu_1).
\end{align}

\subsubsection{$(\frak{g},\frak{g}^\sigma)
 =(\frak{e}_{7(-25)},\frak{e}_{6(-14)}\oplus\frak{so}(2))$}
 We have
\begin{align}
L^{\frak{g}}(4\omega_7)|_{\frak{g}^\sigma}
\simeq\bigoplus_{k=2}^{\infty}
\Bigl(A_{\frak{q}'(6)}((k-8)\mu_6)\boxtimes \bb{C}_{k-2}\Bigr)
\oplus\bigoplus_{k=-1}^{\infty}
\Bigl(A_{\frak{q}'(1)}((k-4)\mu_1)\boxtimes \bb{C}_{-k-2}\Bigr).
\end{align}

\subsubsection{$(\frak{g},\frak{g}^\sigma)
 =(\frak{e}_{7(-25)},\frak{so}(2,10)\oplus\frak{sp}(1,\bb{R}))$}
 We have
\begin{align}
L^{\frak{g}}(4\omega_7)|_{\frak{g}^\sigma}
&\simeq\bigoplus_{k=0}^{\infty}
A_{\frak{q}'(5)}((k-2)\mu_5)\boxtimes
A_{\frak{q}''(1)}((k+2)\nu_1).
\end{align}

\subsubsection{$(\frak{g},\frak{g}^\sigma)
 =(\frak{e}_{7(-25)},\frak{su}(6,2))$}
 We have
\begin{align}
L^{\frak{g}}(4\omega_7)|_{\frak{g}^\sigma}
&\simeq\bigoplus_{k=0}^{\infty} A_{\frak{q}'(4)}((k-2)\mu_4).
\end{align}

\subsubsection{$(\frak{g},\frak{g}^\sigma)
 =(\frak{e}_{7(-25)},\frak{so}^*(12)\oplus\frak{su}(2))$}
 We have
\begin{align}
\label{eq:e7sostaraq}
L^{\frak{g}}(4\omega_7)|_{\frak{g}^\sigma}
&\simeq\bigoplus_{k=0}^{\infty}
A_{\frak{q}'(5)}((k-4)\mu_5)\boxtimes
F^{\frak{su}(2)}(k\nu_1).
\end{align}

\subsubsection{Proofs}

To prove \eqref{eq:umnpqaq} -- \eqref{eq:e7sostaraq} 
 we only need to check that both sides are isomorphic 
 as $\frak{k}^\sigma$-modules by \cite[Lemma 8.7]{Kob07}.
The left hand sides are decomposed into 
 irreducible $\frak{k}$-modules as in \eqref{eq:Ktype}
 and then decomposed into $\frak{k}^\sigma$-modules
 by using the branching laws from $\frak{k}$ to $\frak{k}^\sigma$.
For the right hand sides we use Theorem~\ref{Blattner} and
 Remark~\ref{rem:ktype}.

We illustrate computations in the case
 $(\frak{g},\frak{g}^\sigma)
 =(\frak{e}_{7(-25)},\frak{so}^*(12)\oplus\frak{su}(2))$.
Since the highest root for $\frak{g}$ is $\omega_1$, we have
\[L^{\frak{g}}(4\omega_7)|_\frak{k}\simeq
\bigoplus_{l=0}^{\infty}F^{\frak{k}}(l\omega_1+4\omega_7)\]
by \eqref{eq:Ktype}.
The branching law of $F^{\frak{k}}(l\omega_1+4\omega_7)|_{\frak{k}^\sigma}$
 is given by \cite{Nie} and we have
\[F^{\frak{k}}(l\omega_1+4\omega_7)|_{\frak{k}^\sigma}\simeq
 \bigoplus_{\substack{p+2q+r=l \\ p,q,r\in\bb{Z}_{\geq 0}}}
F^{\frak{k}^\sigma}(p\mu_2+q\mu_4+r\mu_5+4\mu_6+r\nu_1).\]
For the right hand side of \eqref{eq:e7sostaraq},
 we use Theorem~\ref{Blattner}.
Let $\frak{g}'=\frak{so}^*(12)$ and $\frak{g}'=\frak{k}'+\frak{p}'$
 the Cartan decomposition.
Write $\frak{q}'(5)=\frak{l}'_\bb{C}+\frak{u}'$ for Levi decomposition.
Then $\frak{l}'\simeq \frak{su}(5,1)\oplus\frak{u}(1)$, 
 $\frak{l}'\cap\frak{k}'\simeq \frak{su}(5)\oplus\frak{u}(1)^2$, and
 $\frak{u}'\cap\frak{p}'_\bb{C}
 \simeq F^{\frak{l}'\cap\frak{k}'}(\mu_2)$.
Hence $2\rho(\frak{u}'\cap\frak{p}'_\bb{C})=4\mu_5+4\mu_6$ and
\[S(\frak{u}'\cap\frak{p}'_\bb{C})
 \simeq \bigoplus_{p,q\in\bb{Z}_{\geq 0}}
 F^{\frak{l}'\cap\frak{k}'}(p\mu_2+q\mu_4)\]
as an $(\frak{l}'\cap\frak{k}')$-module. 
Therefore,
\begin{align*}
A_{\frak{q}'(5)}((k-4)\mu_5)|_{\frak{k}'}\simeq
\bigoplus_{p,q\in \bb{Z}_{\geq 0}}
F^{\frak{k}'}(p\mu_2+q\mu_4+k\mu_5+4\mu_6).
\end{align*}
As a result, both sides of \eqref{eq:e7sostaraq}
 are isomorphic to
\begin{align*}
\bigoplus_{p,q,r\in \bb{Z}_{\geq 0}}
F^{\frak{k}^\sigma}(p\mu_2+q\mu_4+r\mu_5+4\mu_6+r\nu_1),
\end{align*}
which proves \eqref{eq:e7sostaraq}.

For \eqref{eq:umnpqaq} -- \eqref{eq:unnsostaraq}, \eqref{eq:souaq}, 
\eqref{eq:sostaruaq} -- \eqref{eq:spuaq}, \eqref{eq:e6soaq} -- \eqref{eq:e7sostaraq},
 let $A_{\frak{q}'}(\lambda)$ be a derived functor module appearing
 on the right hand side.
Then we can see that $\lambda$ is weakly fair and
 the nilradical of $\frak{q}'$ is abelian.
Hence $A_{\frak{q}'}(\lambda)$ is an irreducible highest weight module
 by Proposition~\ref{abelnil}.
Using Theorem~\ref{Blattner} we can find a weight $\mu$
 such that $A_{\frak{q}'}(\lambda)\simeq L^{\frak{g}'}(\mu)$.
We can thus verify \eqref{eq:umnpq} -- \eqref{eq:unnsostar}, \eqref{eq:sou}, 
\eqref{eq:sostaru} -- \eqref{eq:spu}, \eqref{eq:e6so} -- \eqref{eq:e7sostar}.

\begin{rem}\label{rem:reducibleaq}
By comparing \eqref{eq:unnspn} and \eqref{eq:unnspnaq}, we see that for $\frak{g}=\frak{sp}(n,\bb{R})$, the module $A_{\frak{q}(1)}(-n\omega_1)$ is reducible. Similarly, from \eqref{eq:so2eso2e-1} and \eqref{eq:so2eso2e-1aq} we see that for $\frak{g}=\frak{so}(2,2n)$, the module $A_{\frak{q}(n+1)}(-2\omega_{n+1})$ is reducible.
\end{rem}

\subsection{The Fock model for $\frak{g}=\frak{so}(2,N)$}\label{subsec:Fock}

For $\frak{g}=\frak{so}(2,N)$ and $\frak{g}^\sigma=\frak{so}(2,M)\oplus\frak{so}(N-M)$ we obtain the branching law of $L^{\frak{g}}(c\zeta)|_{\frak{g}^\sigma}$ using an explicit realization of the minimal holomorphic representation, the Fock model. This model is constructed in \cite{KHMO12} for the minimal holomorphic representation and generalized in \cite{Moe13} to all scalar type unitary highest weight representations. In \cite[Theorem 7.2]{Moe13} the desired branching law is derived. We give a brief outline of the proof. For this we use the same notation as in Settings~\ref{so2e} and \ref{so2o}. Write $\omega_1$ for the corresponding fundamental weight of $\frak{so}(2,N)$ and $\mu_1$ for the one of $\frak{so}(2,M)$.

In \cite{Moe13} the $(\frak{g},K)$-modules $N^{\frak g}(x\omega_1)$ are realized on the space
\[\bb{C}[\frak{p}_-]\simeq\bb{C}[Z_1,\ldots,Z_N]\]
of regular functions on $\frak{p}_-\simeq\bb{C}^N$. The $\frak{g}$-action in this realization is given by regular differential operators up to order two. The crucial operators here are the second-order Bessel operators (see \cite[Sections 1.6, 2.4, 4.3]{Moe13}). The points of unitarity are given by $\{0,\frac{N-2}{2}\}\cup(\frac{N-2}{2},\infty)$.

For $x\in(\frac{N-2}{2},\infty)$ the unique irreducible quotient $L^{\frak g}(x\omega_1)$ is $N^{\frak g}(x\omega_1)$ itself, so
\[L^{\frak g}(x\omega_1) = \bb{C}[Z_1,\ldots,Z_N].\]
For $x=\frac{N-2}{2}$ the irreducible quotient is given by
\[L^{\frak g}\Bigl(\frac{N-2}{2}\omega_1\Bigr)=\bb{C}[Z_1,\ldots,Z_N]/\langle Z_1^2-Z_2^2-\cdots-Z_N^2\rangle=\bb{C}[\bb{X}],\]
the space of regular functions on the variety
\[\bb{X}=\{z\in\bb{C}^N:z_1^2=z_2^2+\cdots+z_N^2\}.\]

We first decompose $L^{\frak g}(\frac{N-2}{2}\omega_1)$ with respect to the action of $\frak{so}(N-M)$ on the last $N-M$ coordinates. Denote by $\cal{H}^k(\bb{C}^{N-M})$ the space of spherical harmonics of $N-M$ variables of degree $k$, viewed as holomorphic polynomials on $\bb{C}^{N-M}$. We have
\[\bb{C}[Z_1,\ldots,Z_N]=\bb{C}[Z_1,\ldots,Z_M]\otimes\bb{C}[Z_{M+1},\ldots,Z_N].\]
Further, every polynomial in $Z_{M+1},\ldots,Z_N$ is the sum of spherical harmonics multiplied with powers of $(Z_{M+1}^2+\cdots+Z_N^2)$. Since in $\bb{C}[\bb{X}]$ we have $Z_{M+1}^2+\cdots+Z_N^2=Z_1^2-Z_2^2-\cdots-Z_M^2$ we obtain
\[\bb{C}[\bb{X}]=\bigoplus_{k=0}^\infty\bb{C}[Z_1,\ldots,Z_M]\otimes\cal{H}^k(\bb{C}^{N-M}).\]
Carefully checking the $\frak{so}(2,M)$-action we find that as $\frak{g}^\sigma$-representations this gives
\[L^{\frak{so}(2,N)}\Bigl(\frac{N-2}{2}\omega_1\Bigr)=\bigoplus_{k=0}^\infty L^{\frak{so}(2,M)}\Bigl(\Bigl(\frac{N-2}{2}+k\Bigr)\mu_1\Bigr)\boxtimes\cal{H}^k(\bb{C}^{N-M}).\]
This proves the formulas \eqref{eq:so2eso2e} -- \eqref{eq:so2oso2o-2}.

\section{Branching laws for non-holomorphic symmetric pairs}\label{sec:nonhol}

By \cite[Theorem 5.2]{KO} the only symmetric pairs $(\frak{g},\frak{g}^\sigma)$ of non-holomorphic type such that $L^{\frak{g}}(c\zeta)|_{\frak{g}^\sigma}$ is discretely decomposable are given by
\begin{equation}
\begin{aligned}
 &(\frak{su}(2m,2n),\frak{sp}(m,n)), && (m,n\geq1), &&(\frak{so}(2,n),\frak{so}(1,n)), && (n\geq3),\\
 &(\frak{sp}(2n,\bb{R}),\frak{sp}(n,\bb{C})), && (n\geq1), &&(\frak{e}_{6(-14)},\frak{f}_{4(-20)}).
\end{aligned}\label{eq:NonHolomorphicPairs}
\end{equation}

\subsection{Irreducibility}\label{sec:nonholirred} We first prove that the restriction is irreducible:

\begin{thm}\label{thm:IrreducibilityForNonHolomorphicPairs}
If $(\frak{g},\frak{g}^\sigma)$ is one of the pairs in \eqref{eq:NonHolomorphicPairs} then the restriction $L^{\frak{g}}(c\zeta)|_{\frak{g}^\sigma}$ is irreducible.
\end{thm}

For the pair $(\frak{so}(2,n),\frak{so}(1,n))$ this was already proved by Sepp\"{a}nen~\cite[Theorem 19]{Sep07i}. He identified the restriction with a complementary series representation of $\frak{so}(1,n)$. For $(\frak{e}_{6(-14)},\frak{f}_{4(-20)})$ the irreducibility was shown by Binegar--Zierau~\cite[Theorem 3.3]{BZ94}. They identify the restriction with a certain Zuckerman's derived functor module. In Section~\ref{sec:idnonhol} we will also identify the two remaining cases with known representations.

The general result in Theorem~\ref{thm:IrreducibilityForNonHolomorphicPairs} will follow from the following statement which was basically used in \cite{BZ94}:

\begin{prop}\label{prop:irred}
Suppose $L^{\frak{g}}(c\zeta)$ is discretely decomposable when restricted to the non-compact symmetric subalgebra $\frak{g}^\sigma$. If each $K$-type $F^{\frak{k}}(c\zeta+k\beta)$ is irreducible when restricted to $\frak{k}^\sigma$ and if they are pairwise non-isomorphic as $\frak{k}^\sigma$-modules, then $L^{\frak{g}}(c\zeta)$ is irreducible when restricted to $\frak{g}^\sigma$.
\end{prop}

\begin{proof}
Assume $X$ is a $\frak{g}^\sigma$-stable subspace of $L^{\frak{g}}(c\zeta)$.  Then by our assumption $X$ is a direct sum of $K$-types of the form $F^{\frak{k}}(c\zeta+k\beta)$ and in particular $X$ is $\frak{k}$-stable.  Since $\frak{p}^\sigma\neq0$, the $\frak{k}$-submodule of $\frak{p}$ generated by $\frak{p}^\sigma$ has to be $\frak{p}$ itself.  Therefore $X$ is also stable under $\frak{p}$ and thus under $\frak{g}$ which implies that $X=0$ or $L^{\frak{g}}(c\zeta)$ by the irreducibility of $L^{\frak{g}}(c\zeta)$ as a $\frak{g}$-representation.
\end{proof}

It remains to show that the $K$-types $F^{\frak{k}}(c\zeta+k\beta)$ are irreducible and pairwise non-isomorphic when restricted to $\frak{k}^\sigma$ in the four cases considered above. 

\subsubsection{$(\frak{g},\frak{g}^\sigma)=(\frak{su}(2m,2n),\frak{sp}(m,n))$}\label{sec:irredsusp}

We have $c\zeta=\omega_{2m}$ by Setting~\ref{sumn} and the highest root is $\beta=\omega_1+\omega_{2m+2n-1}$ so $c\zeta+k\beta=k\omega_1+\omega_{2m}+k\omega_{2m+2n-1}$.
Then by the Borel--Weil Theorem $F^{\frak{k}}(c\zeta+k\beta)$ is realized as holomorphic sections of a line bundle on the partial flag variety $SU(2m)/ U(2m-1)\times SU(2n)/U(2n-1) \simeq \bb{P}_\bb{C}^{2m-1}\times \bb{P}_\bb{C}^{2n-1}$.  Since $Sp(m)\times Sp(n)$ acts transitively on this variety, we have $SU(2m)/U(2m-1) \simeq Sp(m)/(U(1)\times Sp(m-1))$ and $SU(2n)/U(2n-1) \simeq Sp(n)/(U(1)\times Sp(n-1))$.  Moreover, if two characters of $U(2n-1)$ are non-isomorphic with each other, they are still so when restricted to $U(1)\times Sp(n-1)$.  Therefore the Borel--Weil Theorem again implies that $F^{\frak{k}}(c\zeta+k\beta)$ are irreducible and pairwise non-isomorphic as $(\frak{sp}(m)\oplus\frak{sp}(n))$-module.

\subsubsection{$(\frak{g},\frak{g}^\sigma)=(\frak{so}(2,N),\frak{so}(1,N))$}

We have $\frak{k}=\frak{so}(2)\oplus \frak{so}(N)$ and $\frak{k}^\sigma=\frak{so}(N)$. Hence any irreducible $\frak{k}$-module is written as the outer tensor product $V=V_1\boxtimes V_2$, where $V_1$ is an irreducible $\frak{so}(2)$-module and $V_2$ is an irreducible $\frak{so}(N)$-module. Then $V_1$ is one-dimensional and the restriction $V|_{\frak{so}(N)}\simeq V_2$ is irreducible. Further, for $V=F^{\frak{k}}(c\zeta+k\beta)$ we have $V_2\cong\mathcal{H}^k(\bb{R}^N)$, the space of spherical harmonics of degree $k$ on $\bb{R}^N$ and hence different parameters $k$ give non-isomorphic $\frak{so}(N)$-modules $V_2$ if $N\geq 3$.

\subsubsection{$(\frak{g},\frak{g}^\sigma)=(\frak{sp}(2n,\bb{R}),\frak{sp}(n,\bb{C}))$}

We have $c\zeta=\frac{1}{2}\omega_{2n}$ by Setting~\ref{spr} and the highest root is $\beta=2\omega_1$ so $c\zeta+k\beta=2k\omega_1+\frac{1}{2}\omega_{2n}$.
Then by the Borel--Weil Theorem $F^{\frak{k}}(c\zeta+k\beta)$ is realized as holomorphic sections of a line bundle on the partial flag variety $SU(2n)/ U(2n-1)\simeq\bb{P}_\bb{C}^{2n-1}$.  
As in \ref{sec:irredsusp}, we see that $F^{\frak{k}}(c\zeta+k\beta)$ are irreducible and pairwise non-isomorphic as $\frak{sp}(n)$-module.

\subsubsection{$(\frak{g},\frak{g}^\sigma)=(\frak{e}_{6(-14)},\frak{f}_{4(-20)})$}

We note that $\frak{k}=\frak{so}(10)$ and $\frak{k}^\sigma=\frak{so}(9)$.
We have $c\zeta=3\omega_6$ by Setting~\ref{e6(-14)} and the highest root is $\beta=\omega_2$ so $c\zeta+k\beta=3\omega_6+k\omega_2$.
Then by the Borel--Weil Theorem $F^{\frak{k}}(c\zeta+k\beta)$ is realized as holomorphic sections of a line bundle on the partial flag variety $SO(10)/U(5)$.  Since $SO(9)$ acts transitively on this variety, we have $SO(10)/U(5)\simeq SO(9)/U(4)$.  Moreover, if two characters of $U(5)$ are non-isomorphic with each other, they are still so when restricted to $U(4)$.  Therefore we conclude that $F^{\frak{k}}(c\zeta+k\beta)$ are irreducible and pairwise non-isomorphic as $\frak{so}(9)$-module.

\subsection{Identification}\label{sec:idnonhol}

We identify the irreducible restrictions $L^\frak{g}(c\zeta)|_{\frak{g}^\sigma}$ with known representations. For $(\frak{g},\frak{g}^\sigma)=(\frak{sp}(2n,\bb{R}),\frak{sp}(n,\bb{C}))$ the restriction is the even part of the (complex) metaplectic representation of $\frak{sp}(n,\bb{C})$. In the other three cases the restriction is isomorphic to Zuckerman's module. Further, for $(\frak{so}(2,N),\frak{so}(1,N))$ the restriction is a spherical complementary series representation. Moreover, the restrictions for $(\frak{sp}(2n,\bb{R}),\frak{sp}(n,\bb{C}))$ and $(\frak{su}(2n,2n),\frak{sp}(n,n))$ appear in \cite{HKM,Sah95} as representations with minimal Gelfand--Kirillov dimension.

\subsubsection{$\frak{su}(2m,2n)\downarrow\frak{sp}(m,n)$}

Let $(\frak{g},\frak{g}^\sigma)=(\frak{su}(2m,2n),\frak{sp}(m,n))$ with $m\geq n$. We take $\alpha_i$ as in Setting~\ref{sumn}.
We may assume $\sigma\frak{t}=\frak{t}$ and
\begin{align*}
&\sigma\alpha_{i}=\alpha_{2m-i}\ (1\leq i\leq 2m-1), &&
\sigma\alpha_{2m}=-\beta,\\
&\sigma\alpha_{2m+i}=\alpha_{2m+2n-i}\ (1\leq i\leq 2n-1), &&
\sigma=1\text{ on }\frak{g}_{\alpha_m}\text{ and }\frak{g}_{\alpha_{2m+n}}.
\end{align*}
Then $\frak{t}^\sigma$ is a Cartan subalgebra of $\frak{k}^\sigma$ and of $\frak{g}^\sigma$.
We put
\begin{align*}
&\beta_i:=\alpha_i|_{\frak{t}^\sigma}\ (1\leq i\leq m-1),  &&
\beta_m:=(\alpha_m+\alpha_{m+1}+\cdots+\alpha_{2m})|_{\frak{t}^\sigma},\\
&\beta_{m+i}:=\alpha_{2m+i}|_{\frak{t}^\sigma}\ (1\leq i\leq n-1), &&
\beta_{m+n}:=\alpha_{2m+n}|_{\frak{t}^\sigma}.
\end{align*}
Then $\beta_i$ form a set of simple roots for $\frak{g}^\sigma_\bb{C}$
 and the corresponding Dynkin diagram is:
\[\begin{xy}
\ar@{-} (0,0) *+!D{\beta_1} *{\circ}="A";  (10,0)
\ar@{.} (10,0); (20,0) 
\ar@{-} (20,0); (30,0) *+!D{\beta_{m+n-1}} *{\circ}="F"
\ar@{<=} "F"; (43,0) *+!D{\beta_{m+n}} *{\circ}
\end{xy}\]
We have $\theta=1$ on $\frak{g}^\sigma_{\beta_i}$ for $i\neq m$ 
 and $\theta=-1$ on $\frak{g}^\sigma_{\beta_m}$.
Write $\mu_i$ for the corresponding fundamental weights.

\begin{thm}\label{thm:su(2m,2n)TOsp(m,n)}
For the symmetric pair $(\frak{g},\frak{g}^\sigma)
 =(\frak{su}(2m,2n),\frak{sp}(m,n))$,
\begin{align}\label{eq:susp}
L^{\frak{g}}(\omega_{2m})|_{\frak{g}^\sigma}
\simeq
A_{\frak{q}'(1)}(-2n\mu_1),
\end{align}
where $\frak{q}'(1)$ is the maximal parabolic subalgebra of $\frak{g}^\sigma_\bb{C}$ corresponding to $\beta_1$. Further, for $m=n$ the restriction $L^{\frak{g}}(\omega_{2m})|_{\frak{g}^\sigma}$ is isomorphic to the small representation of $\frak{sp}(n,n)$ constructed in \cite{HKM}.
\end{thm}

\subsubsection{$\frak{so}(2,N)\downarrow\frak{so}(1,N)$}
Let $(\frak{g},\frak{g}^\sigma)=(\frak{so}(2,N),\frak{so}(1,N))$.
First assume that $N=2n$ is even. We take $\alpha_i$ as in Setting~\ref{so2e}.
We may assume $\sigma\frak{t}=\frak{t}$ and
\begin{align*}
\sigma\alpha_{1}=-\beta,\quad
\sigma\alpha_{i}=\alpha_{i}\ (2\leq i\leq n+1),\quad
\sigma=1\text{ on }\frak{g}_{\alpha_i}\ (2\leq i\leq n+1).
\end{align*}
Then $\frak{t}^\sigma$ is a Cartan subalgebra of $\frak{k}^\sigma$
 and of $\frak{g}^\sigma$.
We put
\begin{align*}
\beta_i:=\alpha_{i+1}|_{\frak{t}^\sigma}\ (1\leq i\leq n-1),\quad
\beta_n:=(\alpha_{1}+\alpha_2+\cdots+\alpha_{n-1}+\alpha_{n+1})|_{\frak{t}^\sigma}.
\end{align*}
Then $\beta_i$ form a set of simple roots for $\frak{g}^\sigma_\bb{C}$
 and the corresponding Dynkin diagram is:
\[\begin{xy}
\ar@{-} (0,0) *+!D{\beta_1} *{\circ}="A"; 
 (10,0)  *+!D{\beta_2} *{\circ}="B"
\ar@{-} "B";  (20,0)
\ar@{.} (20,0); (30,0) 
\ar@{-} (30,0); (40,0) *+!D{\beta_{n-1}} *{\circ}="F"
\ar@{=>} "F"; (50,0)  *+!D{\beta_n} *{\circ}
\end{xy}\]
We have $\theta=1$ on $\frak{g}^\sigma_{\beta_i}$ for $i<n$ 
 and $\theta=-1$ on $\frak{g}^\sigma_{\beta_n}$.
Write $\mu_i$ for the corresponding fundamental weights.
\begin{thm}
For the symmetric pair $(\frak{g},\frak{g}^\sigma)
 =(\frak{so}(2,2n),\frak{so}(1,2n))$,
\begin{align}\label{eq:so2eso1e}
L^{\frak{g}}((n-1)\omega_1)|_{\frak{g}^\sigma}
\simeq
A_{\frak{q}'(1)}((-n+1)\mu_1)\simeq 
A_{\frak{q}'(n)}(-2\mu_n),
\end{align}
where $\frak{q}'(i)$ is the maximal parabolic subalgebra of $\frak{g}^\sigma_\bb{C}$ corresponding to $\beta_i$. Further, the restriction $L^{\frak{g}}((n-1)\omega_1)|_{\frak{g}^\sigma}$ is isomorphic to a spherical complementary series representation of $\frak{so}(1,2n)$.
\end{thm}

Next assume that $N=2n+1$ is odd. We take $\alpha_i$ as in Setting~\ref{so2o}.
We may assume $\sigma\frak{t}=\frak{t}$ and
\begin{align*}
&\sigma\alpha_{1}=-\beta,\quad
\sigma\alpha_{i}=\alpha_{i}\ (2\leq i\leq n+1),\quad
\sigma=1\text{ on }\frak{g}_{\alpha_i}\ (2\leq i\leq n+1).
\end{align*}
Then $\frak{t}^\sigma$ is a Cartan subalgebra of $\frak{k}^\sigma$.
Let $\frak{h}'$ be the centralizer of $\frak{t}^\sigma$ in $\frak{g}^\sigma$,
 which is a Cartan subalgebra of $\frak{g}^\sigma$.
We define a set of simple roots $\beta_1,\dots,\beta_{n+1}\in(\frak{h}'_\bb{C})^*$ for $\frak{g}^\sigma_\bb{C}$ such that 
\begin{align*}
\beta_i|_{\frak{t}^\sigma}
 =\alpha_{i+1}|_{\frak{t}^\sigma}\ (1\leq i\leq n-1),\quad
\beta_n|_{\frak{t}^\sigma}=\beta_{n+1}|_{\frak{t}^\sigma}
 =\alpha_{n+1}|_{\frak{t}^\sigma}
\end{align*}
and the corresponding Dynkin diagram is:
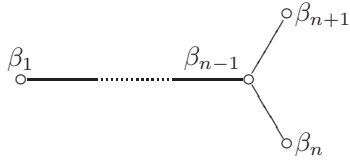
\[\begin{xy}
\ar@{-} (0,0) *+!D{\beta_1} *{\circ}="A";  (10,0)
\ar@{.} (10,0); (20,0) 
\ar@{-} (20,0); (30,0) *+!DR{\beta_{n-1}} *{\circ}="F"
\ar@{-} "F"; (35,8.6)  *+!L{\beta_{n+1}} *{\circ}
\ar@{-} "F"; (35,-8.6)  *+!L{\beta_{n}} *{\circ}
\end{xy}\]
We have $\theta\beta_i=\beta_i$ for $1\leq i\leq n-1$,
 $\theta\beta_n=\beta_{n+1}$ and 
 $\theta=1$ on $\frak{g}^\sigma_{\beta_i}$ for $1\leq i\leq n-1$.
Write $\mu_i$ for the corresponding fundamental weights.

\begin{thm}
For the symmetric pair $(\frak{g},\frak{g}^\sigma)
 =(\frak{so}(2,2n+1),\frak{so}(1,2n+1))$,
\begin{align*}
L^{\frak{g}}\Bigl(\Bigl(n-\frac{1}{2}\Bigr)\omega_1\Bigr)\Bigr|_{\frak{g}^\sigma}
\simeq
A_{\frak{q}'(1)}\Bigl(\Bigl(-n+\frac{1}{2}\Bigr)\mu_1\Bigr),
\end{align*}
where $\frak{q}'(1)$ is the maximal parabolic subalgebra of $\frak{g}^\sigma_\bb{C}$ corresponding to $\beta_1$. Further, the restriction $L^{\frak{g}}((n-\frac{1}{2})\omega_1)|_{\frak{g}^\sigma}$ is isomorphic to a spherical complementary series representation of $\frak{so}(1,2n+1)$.
\end{thm}

\subsubsection{$\frak{sp}(2n,\bb{R})\downarrow\frak{sp}(n,\bb{C})$}
Let $(\frak{g},\frak{g}^\sigma)=(\frak{sp}(2n,\bb{R}),\frak{sp}(n,\bb{C}))$.
The complex metaplectic representation of $\frak{sp}(n,\bb{C})$ is not as well-known as its counterpart for $\frak{sp}(n,\bb{R})$. It can be realized on $L^2(\bb{C}^n)$ and splits into two irreducible pieces, the even and the odd part (see e.g \cite[page 161]{Vog87}). A possible construction is by restricting the metaplectic representation of $\frak{sp}(2n,\bb{R})$ on $L^2(\bb{R}^{2n})\simeq L^2(\bb{C}^n)$ to $\frak{sp}(n,\bb{C})$ whence the first part of the following result is immediate:

\begin{thm}\label{thm:su(2n,R)TOsp(n,C)}
For the symmetric pair $(\frak{g},\frak{g}^\sigma)
 =(\frak{sp}(2n,\bb{R}),\frak{sp}(n,\bb{C}))$ the restriction $L^{\frak{g}}(\frac{1}{2}\omega_{n})|_{\frak{g}^\sigma}$ is isomorphic to the even part of the metaplectic representation of $\frak{sp}(n,\bb{C})$. Further, the restriction is isomorphic to the small representation of $\frak{sp}(n,\bb{C})$ constructed in \cite{HKM}.
\end{thm}

\subsubsection{$\frak{e}_{6(-14)}\downarrow\frak{f}_{4(-20)}$}

Let $(\frak{g},\frak{g}^\sigma)=(\frak{e}_{6(-14)},\frak{f}_{4(-20)})$. We take $\alpha_i$ as in Setting~\ref{e6(-14)}.
We may assume $\sigma\frak{t}=\frak{t}$ and
\begin{align*}
&\sigma\alpha_1=\alpha_1,\quad \sigma\alpha_2=\alpha_5,\quad
\sigma\alpha_3=\alpha_3,\quad\sigma\alpha_4=\alpha_4,\quad
\sigma\alpha_6=-\beta,\\
&\sigma=1\text{ on }\frak{g}_{\alpha_i}\ (i=1,3,4).
\end{align*}
Then $\frak{t}^\sigma$ is a Cartan subalgebra of $\frak{k}^\sigma$ and of $\frak{g}^\sigma$.
We put 
\begin{align*}
\beta_1:=\alpha_3|_{\frak{t}^\sigma},\quad
\beta_2:=\alpha_4|_{\frak{t}^\sigma},\quad
\beta_3:=\alpha_5|_{\frak{t}^\sigma},\quad
\beta_4:=(\alpha_1+\alpha_3+\alpha_4+\alpha_5+\alpha_6)|_{\frak{t}^\sigma},
\end{align*}
and the corresponding Dynkin diagram is:
\[\begin{xy}
\ar@{-} (0,0) *+!D{\beta_1} *{\circ}="A";  (10,0) *+!D{\beta_2} *{\circ}="B"
\ar@{=>} "B"; (20,0)  *+!D{\beta_3} *{\circ}="C"
\ar@{-} "C"; (30,0) *+!D{\beta_4} *{\circ}="D"
\end{xy}\]
We have $\theta=1$ on $\frak{g}^\sigma_{\beta_i}$ for $i=1,2,3$ 
 and $\theta=-1$ on $\frak{g}^\sigma_{\beta_4}$.
Write $\mu_i$ for the corresponding fundamental weights.

\begin{thm}[{\cite[Theorem 3.4]{BZ94}}]\label{thm:e6(-14)TOf4(-20)}
For the symmetric pair $(\frak{g},\frak{g}^\sigma)
 =(\frak{e}_{6(-14)},\frak{f}_{4(-20)})$,
\begin{align*}
L^{\frak{g}}(3\omega_6)|_{\frak{g}^\sigma}
\simeq
A_{\frak{q}'(1)}(-2\mu_1),
\end{align*}
where $\frak{q}'(1)$ is the maximal parabolic subalgebra of $\frak{g}^\sigma_\bb{C}$ corresponding to $\beta_1$.
\end{thm}

\subsubsection{Proofs}

The fact that the restrictions for $(\frak{so}(2,N),\frak{so}(1,N))$ are spherical complementary series was proved by Sepp\"{a}nen~\cite[Theorem 19]{Sep07i}.

Let us next treat the identifications with $A_\frak{q}(\lambda)$. The case $(\frak{e}_{6(-14)},\frak{f}_{4(-20)})$ was treated in \cite{BZ94}. Their proof can be applied to the other cases as well. Let e.g $(\frak{g},\frak{g}^\sigma)=(\frak{su}(2m,2n),\frak{sp}(m,n))$. Using the argument in \cite{BZ94}, we see that the restriction $L^{\frak{g}}(\omega_{2m})|_{\frak{g}^\sigma}$ has infinitesimal character $\mu_{2n}-\rho_{\frak{g}^\sigma}$. Here, $\rho_{\frak{g}^\sigma}$ is half the sum of positive roots in $\frak{g}^\sigma_\bb{C}$. Since $\mu_{2n}-\rho_{\frak{g}^\sigma}$ and $-2n\mu_1+\rho_{\frak{g}^\sigma}$ lie in the same Weyl group orbit, both sides of \eqref{eq:susp} have the same infinitesimal character. By Theorem~\ref{Blattner}, 
\begin{equation*}
A_{\frak{q}'(1)}(-2n\mu_1)|_{\frak{k}^\sigma}
\simeq \bigoplus_{k=0}^{\infty} F^{\frak{k}^\sigma}(k\mu_1+k\mu_{m+1}).
\end{equation*}
Hence by the proof of Proposition~\ref{prop:irred}, $A_{\frak{q}'(1)}(-2n\mu_1)$ is irreducible. Therefore, $L^{\frak{g}}(\omega_{2m})|_{\frak{g}^\sigma}$ and $A_{\frak{q}'(1)}(-2n\mu_1)$ are spherical irreducible $(\frak{g}^\sigma,K^\sigma)$-modules and have the same infinitesimal character, which implies that they are isomorphic (see \cite[\S 7]{Hel64}). The case $(\frak{g},\frak{g}^\sigma)=(\frak{so}(2,N),\frak{so}(1,N))$ can be proved in the same way.

We note that for the pairs $(\frak{su}(2m,2n),\frak{sp}(m,n))$ and $(\frak{so}(2,2n),\frak{so}(1,2n))$, the formulas \eqref{eq:susp} and \eqref{eq:so2eso1e} can also be derived by using ${\cal D}$-modules (see \cite{Os11}).

Finally, the identification with representations studied in \cite{HKM} works in the same way as the identification with $A_\frak{q}(\lambda)$. Since both representations in question are spherical irreducible representations of classical groups, it suffices by \cite[\S 7]{Hel64} to show that their infinitesimal characters agree. This can be done using \cite[Theorem 3.7]{HKM}.

\section{Associated varieties and Kobayashi's conjecture}\label{sec:AssVar}

We study a conjecture by Kobayashi for the associated varieties of discrete components in the restriction of $\frak{g}$-representations. We confirm the conjecture for all discretely decomposable restrictions of minimal holomorphic representations to symmetric subgroups. 

Let $G$ be a real reductive group and $G'$ a reductive subgroup.
Take a maximal compact subgroup $K$ of $G$ such that $K':=G'\cap K$ is a maximal compact subgroup of $G'$.
For a $\frak{g}$-module $X$ of finite length we denote by $\cal{V}_{\frak{g}_{\bb C}}(X)\subseteq\frak{g}_{\bb C}^*$ its associated variety in the sense of Vogan~\cite{Vog91}. Accordingly we will use $\cal{V}_{\frak{g}'_{\bb C}}(Y)\subseteq(\frak{g}'_{\bb C})^*$ for the associated variety of a $\frak{g}'$-module $Y$ of finite length. Let $\pr_{\frak{g}\to\frak{g}'}:\frak{g}_{\bb C}^*\to(\frak{g}'_{\bb C})^*$ denote the restriction dual to the embedding $\frak{g}'_{\bb C}\to\frak{g}_{\bb C}$. Then Kobayashi conjectured:

\begin{conj}[{\cite[Conjecture 5.11]{Kob11}}]\label{conj:Kobayashi}
Let $X$ be an irreducible unitarizable $(\frak{g},K)$-module and $Y$ an irreducible $(\frak{g}',K')$-module. If $\Hom_{\frak{g}'}(Y,X)\neq0$ then
\begin{align}\label{eq:assvar}
 \pr_{\frak{g}\to\frak{g}'}(\cal{V}_{\frak{g}_{\bb C}}(X)) = \cal{V}_{\frak{g}'_{\bb C}}(Y). 
\end{align}
\end{conj}

The inclusion $\pr_{\frak{g}\to\frak{g}'}(\cal{V}_{\frak{g}_{\bb C}}(X)) \subset \cal{V}_{\frak{g}'_{\bb C}}(Y)$
was proved in \cite[Theorem 3.1]{Kob98ii}.  The other inclusion is known to hold in the following cases.

\begin{prop}[{\cite[Proposition 5.12]{Kob11}}]\label{prop:Kobayashi}
Conjecture \ref{conj:Kobayashi} is true for the following four cases:
\begin{enumerate}
\item $X$ is the oscillator representation of $\frak{g}=\frak{sp}(n,\bb{R})$ and $\frak{g}'=\frak{g}'_1\oplus\frak{g}'_2$ is 
 a compact dual pair in $\frak{g}$,
\item $X$ is the underlying $(\frak{g},K)$-module of the minimal representation of $O(p,q)$ with $p+q$ even and
 $(\frak{g},\frak{g}')$ is a symmetric pair (\cite{KO03}),
\item $X$ is a (generalized) Verma module and $(\frak{g},\frak{g}')$ is a symmetric pair (\cite{Kob12}),
\item $X=A_\frak{q}(\lambda)$ and $(\frak{g},\frak{g}')$ is a symmetric pair (\cite{Os11}).
\end{enumerate}
\end{prop}

If $X$ is an irreducible $(\frak{g},K)$-module,  $\cal{V}_{\frak{g}_{\bb C}}(X)\subseteq\frak{p}_{\bb C}^*$ is a union of a finite number of nilpotent $K_{\bb C}$-orbits in $\frak{p}_\bb{C}^*$. Identifying $\frak{p}_{\bb C}\simeq\frak{p}_{\bb C}^*$ by means of the Killing form we can view $\cal{V}_{\frak{g}_{\bb C}}(X)$ as a subvariety of $\frak{p}_{\bb C}$. If additionally $X$ is a lowest weight module then $\cal{V}_{\frak{g}_{\bb C}}(X)\subseteq\frak{p}_-$.

\subsection{The conjecture for holomorphic pairs}

We will prove that the conjecture is true if $X$ is a highest (or lowest) weight module and $(\frak{g},\frak{g}')$ is holomorphic in the following sense:
\begin{de}\label{de:holpair2}
Suppose that $\frak{g}$ and $\frak{g}'$ are of Hermitian type.
We say a pair $(\frak{g},\frak{g}')$ is {\it holomorphic} if the natural embedding $G'/K'\to G/K$ is holomorphic.
\end{de}
Note that the pair $(\frak{g},\frak{g}')$ is not assumed to be symmetric as in Definition~\ref{de:holpair}. If we write the decompositions of the tangent spaces of $G/K$ and $G'/K'$
 as $\frak{p}_\bb{C}=\frak{p}_++\frak{p}_-$ and $\frak{p}'_\bb{C}=\frak{p}'_++\frak{p}'_-$, respectively, then the condition for $(\frak{g},\frak{g}')$ to be holomorphic is equivalent to $\frak{p}'_+=\frak{p}'_\bb{C}\cap \frak{p}_+$ and $\frak{p}'_-=\frak{p}'_\bb{C}\cap \frak{p}_-$. Hence, Definition~\ref{de:holpair2} is compatible with Definition~\ref{de:holpair} for a symmetric pair.

If $(\frak{g},\frak{g}')$ is holomorphic, then any highest (or lowest) weight $(\frak{g},K)$-module $X$ is discretely decomposable as a $\frak{g}'$-module and we have:
\begin{thm}\label{thm:assvar}
Let $(\frak{g},\frak{g}')$ be a holomorphic pair and $X$ a highest (or lowest) weight $(\frak{g},K)$-module. 
Then Conjecture~\ref{conj:Kobayashi} is true.
\end{thm}

The key ingredient for the proof is that \eqref{eq:assvar} follows from the `compatibility of filtrations': 
\begin{lem}\label{lem:compatifilt}
Let $X_0$ be a non-zero finite-dimensional subspace of $X$.
Suppose that $U_n(\frak{g}_\bb{C})X_0\cap U(\frak{g}'_\bb{C})X_0=U_n(\frak{g}'_\bb{C})X_0$ for any $n\in \bb{N}$,
 where $U_n(\frak{g}_\bb{C})$ and $U_n(\frak{g}'_\bb{C})$ are given by standard filtrations of $U(\frak{g}_\bb{C})$ and $U(\frak{g}'_\bb{C})$, respectively.
Then the equality~\eqref{eq:assvar} holds.
\end{lem}

\begin{proof}
Since $\pr_{\frak{g}\to\frak{g}'}(\cal{V}_{\frak{g}_{\bb C}}(X)) \subset \cal{V}_{\frak{g}'_{\bb C}}(Y)$  by {\cite[Theorem 3.1]{Kob98ii}}, it suffices to prove the other inclusion $\pr_{\frak{g}\to\frak{g}'}(\cal{V}_{\frak{g}_{\bb C}}(X)) \supset \cal{V}_{\frak{g}'_{\bb C}}(Y)$.
Put $X'=U(\frak{g}'_\bb{C})X_0$.
Define filtrations of $X$ and $X'$ by $X_n=U_n(\frak{g}_\bb{C})X_0$ and $X'_n=U_n(\frak{g}'_\bb{C})X_0$, respectively.  Then our assumption implies that the induced map between graded modules ${\rm gr}\, X'\to {\rm gr}\, X$ is injective.
Hence $\Ann_{S(\frak{g}'_\bb{C})}({\rm gr}\, X')\supset \Ann_{S(\frak{g}_\bb{C})}({\rm gr}\, X)\cap S(\frak{g}'_\bb{C})$ and the inclusion $\pr_{\frak{g}\to\frak{g}'}(\cal{V}_{\frak{g}_{\bb C}}(X)) \supset \cal{V}_{\frak{g}'_{\bb C}}(X')$ holds.
Therefore, if $X''$ is an irreducible $\frak{g}'$-submodule of $X'$, we have $\pr_{\frak{g}\to\frak{g}'}(\cal{V}_{\frak{g}_{\bb C}}(X)) \supset \cal{V}_{\frak{g}'_{\bb C}}(X'')=\cal{V}_{\frak{g}'_{\bb C}}(Y)$ by \cite[Theorem 3.7]{Kob98ii}.
\end{proof}

\begin{proof}[Proof of Theorem~\ref{thm:assvar}]
Suppose that $X$ is an irreducible lowest weight module.
For $a\in \bb{C}$, write $X(a):=\{v\in X:z'v=av\}$, the $z'$-eigenspace with eigenvalue $a$. Then since $\ad(z')$ is $1$ on $\frak{p}_+$ we have $\frak{p}_+X(a)=X(a+1)$.
Put $X_0:=X^{\frak{p}_-}$ and $X_n=U_n(\frak{g}_\bb{C})X_0$ for $n\in \bb{N}$.
Since $X$ is irreducible, $\ad(z')$ acts on $X^{\frak{p}_-}$ by a scalar, say $a_0$. Then we have
\[X_n=U_n(\frak{g}_\bb{C})X^{\frak{p}_-}=U_n(\frak{p}_+)X^{\frak{p}_-}=\bigoplus_{k=0}^n X(a_0+k).\]
Similarly, 
\[U_n(\frak{g}'_\bb{C})X^{\frak{p}_-}=U_n(\frak{p}'_+)X^{\frak{p}_-}=\bigoplus_{k=0}^n X(a_0+k)\cap U(\frak{g}'_\bb{C})X^{\frak{p}_-}.\]
Therefore, $U_n(\frak{g}_\bb{C})X_0\cap U(\frak{g}'_\bb{C})X_0=\bigoplus_{k=0}^n X(a_0+k) \cap U(\frak{g}'_\bb{C})X_0 = U_n(\frak{g}'_\bb{C})X_0$.
Hence \eqref{eq:assvar} follows from Lemma~\ref{lem:compatifilt}.
\end{proof}

\subsection{The conjecture for non-holomorphic pairs}

We next consider the setting in Section~\ref{sec:nonhol}. In particular, $X=L^{\frak{g}}(c\zeta)$ is the minimal holomorphic representation.
Put $X_0:=X^{\frak{p}_-}$ and $X_n:=U_n(\frak{g}_\bb{C})X_0$.
Then $X_n=\bigoplus_{k=0}^n F^{\frak{k}}(c\zeta+k\beta)$.
In view of Proposition~\ref{prop:irred} and its proof, we inductively get $U_n(\frak{g}'_\bb{C})X_0=\bigoplus_{k=0}^n F^{\frak{k}}(c\zeta+k\beta)$.  Hence the assumption in Lemma~\ref{lem:compatifilt} is satisfied.  We thus obtain:
\begin{thm}
Suppose that $(\frak{g},\frak{g}^\sigma)$ is a non-holomorphic symmetric pair, $X=L^{\frak{g}}(c\zeta)$ is the minimal holomorphic representation, and $X|_{\frak{g}^\sigma}$ is discretely decomposable.
Then Conjecture~\ref{conj:Kobayashi} is true.
\end{thm}

\end{document}